\documentclass{amsart}
\usepackage{amsmath,amssymb}
\usepackage{graphicx}
\newtheorem{theorem}{Theorem}[section]
\newtheorem{proposition}[theorem]{Proposition}
\newtheorem{corollary}[theorem]{Corollary}

\newtheorem{lemma}[theorem]{Lemma}

\theoremstyle{definition}
\newtheorem{definition}[theorem]{Definition}

\theoremstyle{remark}
\newtheorem{remark}[theorem]{Remark}

\numberwithin{equation}{section}

\newcommand{\al}{\alpha}
\newcommand{\be}{\beta}
\newcommand{\de}{\delta}
\newcommand{\ep}{\epsilon}

\newcommand{\ga}{\gamma}

\newcommand{\si}{\sigma}
\newcommand{\te}{\theta}
\newcommand{\vp}{\varphi}

\newcommand{\De}{\Delta}
\newcommand{\Ga}{\Gamma}

\newcommand{\bD}{\mathbf{D}}
\newcommand{\bH}{\mathbf{H}}
\newcommand{\bL}{\mathbf{L}}
\newcommand{\tX}{\tilde{\mathcal X}}

\newcommand{\trho}{\tilde{\rho}}
\newcommand{\tN}{\widetilde{\cN}}

\def\NN{\mathbb{N}}
\def\RR{\mathbb{R}}

\def\HH{\mathbb{H}}

\renewcommand\SS{\mathbb{S}}

\newcommand{\cB}{{\mathcal B}}

\newcommand{\cK}{{\mathcal K}}
\newcommand{\cM}{{\mathcal M}}
\newcommand{\cN}{{\mathcal N}}
\newcommand{\cO}{{\mathcal O}}

\newcommand{\cT}{{\mathcal T}}
\newcommand{\cU}{{\mathcal U}}
\newcommand{\cV}{{\mathcal V}}

\newcommand\cX{\mathcal X}

\newcommand{\pd}{\partial}
\newcommand\minus\backslash

\newcommand{\ms}{\mspace{1mu}}

\newcommand{\e}{{\mathrm e}}

\renewcommand\leq\leqslant
\renewcommand\geq\geqslant
\newlength{\intwidth}

\newcommand\AdS{\mathrm{AdS}_n}
\newcommand\cpt{_{\mathrm{c}}}
\newcommand\loc{_{\mathrm{loc}}}
\newcommand\pol{_{\mathrm{p}}}

\begin{document}

\title[Wave equations and holography in asymptotically AdS spaces]{A singular initial-boundary value
  problem for nonlinear wave equations and holography in asymptotically anti-de
  Sitter spaces}

\author{Alberto Enciso}
\address{Instituto de Ciencias Matem\'aticas, Consejo Superior de
  Investigaciones Cient\'\i ficas, 28049 Madrid, Spain}
\email{aenciso@icmat.es}

\author{Niky Kamran}
\address{Department of Mathematics
  and Statistics, McGill University, Montr\'eal, Qu\'ebec, Canada H3A 2K6}
\email{nkamran@math.mcgill.ca}

%
%
\begin{abstract}
  We analyze the initial value problem for semilinear wave equations
  on asymptotically anti-de Sitter spaces using energy methods adapted
  to the geometry of the problem at infinity. The key feature is that
  the coefficients become strongly singular at infinity, which leads
  to considering nontrivial data on the conformal boundary of the
  manifold. This question arises in Physics as the holographic
  prescription problem in string theory.
\end{abstract}
\maketitle
\setcounter{tocdepth}{1}
\tableofcontents

\section{Introduction}
\label{S:intro}


Our goal in this paper is to study certain kinds of semilinear wave
equations with non-constant coefficients that are of interest in string
theory. These equations are defined by the fact that their principal
part is close, in a certain sense, to the wave operator of the
$n$-dimensional anti-de Sitter (AdS) space, and their salient feature is that the
coefficients become very singular at infinity, so that the problem can
be thought of as including some (nontrivial) boundary conditions at
infinity. The analysis of the effect of the associated ``boundary data
at infinity'' on the solutions of the wave equation, which is the
main theme of this paper, is key in the analysis of these wave
equations and, as shall see later on, has a very direct physical
interpretation in the context of the celebrated AdS/CFT correspondence~\cite{Maldacena}.

More specifically, we shall analyze wave equations of the form
\begin{equation}\label{LW}
\square_g\phi-\mu\ms\phi=0\,,
\end{equation}
and their semilinear analog
\begin{equation}\label{NLW}
\square_g\phi-\mu\ms\phi=F(\nabla\phi)\,,
\end{equation}
in Lorentzian manifolds whose geometry at infinity is close to that of
the $n$-dimensio-\break nal AdS space $\AdS$. These manifolds are called
asymptotically AdS, and their precise definition will be given in
Section~\ref{S.LW}. In the equations,  $\mu$ is a real constant,
$\nabla$~stands for the space-time gradient and we omit notationally
the possible dependence of the nonlinearity $F(\nabla\phi)$ on the
space-time variables.

To give a clearer idea of the structure of these equations, and also to set up some notation, it is useful to first recall the definition of AdS space as well as the form of the scalar wave equation therein.
The $n$-dimensional AdS space is the simply connected Lorentzian space of constant negative sectional curvature $-K$. The metric in $\AdS$ is thus given by
\begin{equation*}
g_{AdS_{n}}=-K^{-2}\cosh^2 (Kr)\, dt^2+dr^2+K^{-2}\sinh^2 (Kr)\, g_{\SS^{n-2}}\,,
\end{equation*}
where $t\in\RR$ is the time variable, $r\in\RR^+$,
$\theta\in\SS^{n-2}$ and $g_{\SS^{n-2}}$ is the metric of the unit
$(n-2)$-sphere. In what follows we will take $K:=1/2$.

To analyze the wave equation in $\AdS$, it is convenient to replace the radial coordinate
$r$ by a certain function thereof, $x$, which takes values in $(0,1]$
and is defined through the relation
\[
x:=\bigg(2\cosh \frac r2-1\bigg)^{-\frac12}\,.
\]
In terms of this variable, the AdS metric is given by
\begin{equation}\label{AdSx}
g_{AdS_{n}}= \frac{-(1+x^2)^2\,d t^2+d x^2+(1-x^2)^2\, g_{\SS^{n-2}}}{x^2}\,,
\end{equation}
and the wave equation~\eqref{LW} in $\AdS$ takes the form
\begin{equation}\label{LWAdS2}
-\pd_{t}^2\phi+\pd_{x}^2\phi-\frac{n-2 }x\, \pd_x \phi
+\De_\te \phi-\frac{\mu}{x^2}\, \phi+\cdots=0\,,
\end{equation}
where $\De_\te$ is the Laplacian in $\SS^{n-2}$ and the dots stand for
terms that are smaller, in a certain sense, 
for $x$ close to zero.

The point $x=1$ corresponds to the
spatial origin $r=0$ while the set $\{x=0\}$ (which is a cylinder
$\RR_t\times\SS^{n-2}_\theta$ and is often called the {\em conformal
boundary}\/ or {\em conformal infinity}\/ of the manifold) corresponds to the
spatial infinity of the manifold. Notice that if we conformally
rescale the metric~(\ref{AdSx}) by the factor $x^2$ so that it becomes
smooth when $x=0$, and then compactify the manifold by adding the
boundary set $\{x=0\}$, then this boundary, which is homeomorphic to a
cylinder, is a time-like hypersurface in the extended manifold endowed
with the Lorentzian metric $-dt^2+g_{\SS^{n-2}}$.

The coefficients of the equation are
obviously singular at $x=0$. A simple analysis of the singularities reveals
that the solutions to~\eqref{LWAdS2} are expected to behave at conformal
infinity as
\[
x^{\frac{n-1}2+\al}\big[1+\cO(x)\big]\quad \text{or}\quad x^{\frac{n-1}2-\al}\big[1+\cO(x)\big]\,,
\]
where we hereafter write
\begin{equation}\label{al}
\al:=\bigg[\bigg(\frac{n-1}2\bigg)^2+\mu\bigg]^{1/2}\,.
\end{equation}
Throughout this paper we will assume that 
\begin{equation}\label{threshold}
\mu>-\bigg(\frac{n-1}2\bigg)^2\,,
\end{equation}
so that $\al$ is strictly positive.  Let us mention in passing that
when $\al<1$, both expansions at infinity are square-integrable, which
means that one needs to impose certain boundary conditions at infinity
to determine the evolution of the wave equation. This is directly
related with the fact that $\AdS$, just as any other asymptotically
AdS space, is not globally hyperbolic because null geodesics escape to
infinity for finite values of their affine parameter. This has the
effect that the evolution under the wave equation~\eqref{LW} in $\AdS$
of smooth, compactly supported initial data does not remain compactly
supported in space for all values of time.  It should be stressed,
however, that we will be interested in the boundary behavior for all
values of the parameter $\al$, not just for $\al<1$.


Our goal in this paper is to analyze how solutions to the
Eqs.~\eqref{LW}--\eqref{NLW} in an asymptotically AdS space $\cM$
are determined by prescribing initial conditions {\em and}\/ a (compatible)
boundary condition at infinity. We are mostly interested in the role
of the latter, as this is the key property of this equation. To
prescribe these boundary conditions at conformal infinity, we pick the
dominant exponent and in $\AdS$ it is enough to impose 
\[
x^{\al-\frac{n-1}2}\phi|_{x=0}=f(t,\theta)\,,
\]
where $f$ is the datum of the problem. In an arbitrary asymptotically
AdS manifold $\cM$ there is a direct counterpart of this condition,
\begin{equation}\label{BCs2}
x^{\al-\frac{n-1}2}\phi|_{\pd\cM}=f\,,
\end{equation}
but we will not make this precise yet. Of course, the
initial conditions we take are of the form
\[
\phi|_{\cM_0}=\phi_0\,,\qquad \cT\phi|_{\cM_0}=\phi_1\,,
\]
where $\cM_0$ is a time slice and the vector field $\cT$ roughly
encodes the partial derivative with respect to time (details will be given in
Section~\ref{S.LW}). For the purposes of this Introduction, one
can take trivial initial conditions $\phi_0=\phi_1=0$.

As we have mentioned, the problem under consideration arises naturally
in the context of the AdS/CFT correspondence in string
theory~\cite{Maldacena}, so we shall elaborate a little on this
topic. This is a conjectural relation which posits that a
gravitational field on a Lorentzian $n$-manifold endowed with an
Einstein metric metric close to the $\AdS$ metric at infinity can be
recovered from a gauge field defined on the conformal boundary of the
manifold. The gravitational field is modeled using some PDE of
hyperbolic character in the manifold (typically, the Einstein equation) and
the gauge field at conformal infinity plays the role of boundary datum
through a relation analogous to~\eqref{BCs2}. Since, in harmonic
coordinates, the Einstein equation reduces to a nonlinear wave
equation, at the heart of the AdS/CFT correspondence lies a
boundary-value problem closely related to the one we have stated above
for the scalar wave equation~\eqref{NLW}. Indeed, the {\em holographic
  principle}\/ asserts that the boundary data (which here, in the
context of scalar wave equations, would be the function $f$), defined on
defined on an $(n-1)$-dimensional boundary, propagates through a
suitable $n$-manifold (which is referred to as the {\em bulk}\/ and is
denoted here by $\cM$) to determine the field (here $\phi$) via a
locally well posed problem.

It is important at this stage to note that most rigorous results
related to the holographic principle have been obtained for the
elliptic counterparts of these equations. As is well known, the Riemannian analog
of $\AdS$ is the hyperbolic space $\HH^n$, whose metric we describe using
coordinates $r\in[0,1)$ and $\theta\in\SS^{n-1}$ as
\[
g_{\HH^n}=\frac{dr^2+r^2g_{\SS^{n-1}}}{(1-r^2)^2}\,.
\]
Setting $x:=1-r^2$, it is classical that the corresponding boundary value problem
\begin{equation}\label{Riem}
\De_{{\HH^n}}\phi-\mu\ms\phi=0\,,\qquad x^{\al-\frac{n-1}2}\phi|_{x=0}=f(\theta)\,,
\end{equation}
has a unique solution. Lacking results on the Lorentzian analog of
this problem, the boundary value problem~ (\ref{Riem}) has been used
as the basic model to understand holographic prescription, starting
with Witten's paper~\cite{Witten}.  Of course, the results for the
hyperbolic space remain valid in much more general contexts. For
instance, in the case $\mu=0$, they have been extended by
Anderson~\cite{Anderson83}, Sullivan~\cite{Sullivan} and Anderson and
Schoen~\cite{AS} to treat harmonic functions on manifolds whose
sectional curvature is not assumed to be a negative constant, as is
the case of $\HH^n$, but is pinched between two distinct negative
constants.

For the full Einstein equations with Riemannian signature, the
situation is already much more subtle~\cite{Anderson05}. A fundamental result in this
direction, due to Graham and Lee~\cite{GL}, states that given a Riemannian
metric $g_0$ on the sphere close enough to the round metric, there is
a unique asymptotically hyperbolic metric in the ball close to $\HH^n$
that has $g_0$ as its boundary value, in a suitable sense. In
Riemannian signature, this shows how
the holographic principle works when the boundary datum is small,
which is the case of interest in cosmological applications.
Significant refinements of the Graham--Lee theorem can be found
in~\cite{Anderson08} and references therein.

The situation for the holographic prescription problem in Lorentzian
signature is much less clear-cut, since both the available analytical techniques  
and the expected results are necessarily different. To the best of our knowledge, the wave equation on AdS$_4$ was first
considered in~\cite{BF}, where separation of variables was used to
discuss the behavior of the energy for all values of $\mu$ above the
threshold value~\eqref{threshold}. Again for AdS$_4$,
Choquet-Bruhat~\cite{Choquet1,Choquet2} proved global existence for
the Yang--Mills equation under a radiation condition, and Ishibashi
and Wald gave a proof of the well-posedness of the Cauchy problem for
the Klein--Gordon equation~\eqref{LW} in $\AdS$ using spectral
theory. More refined results for the Klein--Gordon equation in an AdS space were
developed by Bachelot~\cite{Bachelot1,Bachelot2,Bachelot3}, who used energy
methods and dispersive estimates to study the decay of the solutions
and prove Strichartz estimates and some results on the propagation of
singularities.

An important paper is due to Vasy,~\cite{Vasy}, where fine results on the
propagation of singularities are proved for the Klein--Gordon equation
on asymptotically AdS spaces using microlocal analysis. Holzegel and
Warnick, both independently and in joint
work~\cite{Holzegel,Warnick,HW}, used energy methods to show that the
Cauchy problem for this equation in asymptotically AdS$_4$ space-times is well posed in twisted
Sobolev spaces (considering different
``homogeneous boundary conditions'' when the parameter $\mu$ is
negative and above its threshold value), and discussed the uniform
boundedness of solutions to the Klein--Gordon equation in stationary AdS black
hole geometries.


In this paper, we prove that the boundary value problem at infinity for the wave
equations~\eqref{LW}--\eqref{NLW} in an asymptotically AdS manifold is
well posed in a certain scale of Sobolev spaces adapted to the
geometry of the space-time. In the language of physics, this can be
rephrased as saying that the holographic prescription problem is well
posed for scalar fields with suitable nonlinearities. The prescription
has the physically critical properties of being {\em fully holographic}\/ and
{\em causal}\/, which essentially means that, for trivial initial conditions
and compactly supported datum $f$ on the conformal boundary, the field
$\phi$ is purely controlled by $f$ and is identically zero for all
times below the support of this function~\cite{PRD}.

The precise statements of the well-posedness results proved in this
paper are given in Theorems~\ref{T.LW} and~\ref{T.nonlinear}. Although
we will not reproduce these statements here to avoid introducing too
much notation, we shall nevertheless discuss the content of these theorems in some
detail. 

For the linear wave equation, we use energy estimates to prove
that, given (compatible) initial conditions and a datum on the
conformal boundary $f$, there is a unique solution of~\eqref{LW}. This
solution is defined for all time and can be estimated in terms of the
data in suitable Sobolev spaces with norms that are defined in a
way that compensates for the singular behavior of the metric at the
conformal boundary. An interesting feature is that, due to the form
of the energy, one does not simply get the usual weighted Sobolev
spaces, but rather a twisted version thereof, given in~(\ref{bH}),
involving both a weight vanishing at the conformal boundary of the
manifold and twisted derivatives, where the twist factor conjugating
the derivative (see~(\ref{twist})) is directly related to the geometry
of the asymptotically AdS space at infinity through the ``renormalized
energy'' considered by Breitenlohner and Freedman~\cite{BF}. One also
shows that these Sobolev-type estimates imply the pointwise decay of
the solutions by proving suitable Sobolev embedding theorems. 

We would like to point out at this stage that Warnick~\cite{Warnick}
has proved in the range $\alpha < 1 $ a well-posedness result of a
nature similar to that of Theorem~\ref{T.LW}, obtaining first-order
Sobolev estimates relating the norm of the unique weak solution of the
initial-boundary value problem for the Klein--Gordon equation to the
norms of the initial and boundary data. The results that we prove in
Theorem~\ref{T.LW} are somewhat stronger, as they apply to the full
range $\alpha>0$ and involve estimates on the higher-order Sobolev
norms of the solution and its time derivatives. These stronger results
turn out to be of crucial importance for the proof we give in
Theorem~\ref{T.nonlinear} of the well-posedness for non-linear wave
equations. Vasy's microlocal approach to wave equations on
asymptotically AdS spacetimes~\cite{Vasy} also yields a closely
related result for the Klein--Gordon equation, together with more
powerful results on the propagation of singularities. However, the use
of twisted Sobolev spaces together with suitable embedding theorems
seems to be better suited for the analysis of nonlinear wave equations
on these spacetimes, which is the ultimate goal of this paper.

Indeed, for the nonlinear wave equation we prove a local well-posedness
result (with analogous estimates) using a bootstrap
argument. Although the argument is applicable to more general
nonlinearities, for concreteness we restrict our attention to
quadratic nonlinearities of the form
\begin{subequations}\label{nonlinF}
\begin{equation}
F(\nabla\phi):=\Ga\, g(\nabla\phi,\nabla\phi)\,,
\end{equation}
where the function $\Ga$ behaves in a neighborhood of the conformal
boundary as
\begin{equation}
\Ga=x^q \,\widehat\Ga(t,x,\te)\,,
\end{equation}
\end{subequations}
where $q$ is a large enough power and the function $\widehat\Ga$ is
smooth up to the conformal boundary. (Details will be given in Section~\ref{Nonlinappl}.) 
The reason to consider a nonlinearity quadratic in the first
derivatives is that the problem becomes a simplified version of the
Einstein equation in which, in particular, the tensorial structure is
disregarded. This is of particular relevance for the
proof a Lorentzian analog of the Graham--Lee theorem for
asymptotically AdS Einstein metrics, which will be considered elsewhere~\cite{EK}.

The proofs of these well-posedness results make up the substance of
the rest of our paper. First of all, in Section~\ref{S.defs} we
provide precise definitions of most of the concepts that we have
briefly described in this Introduction and devote Section~\ref{S.Sobolev} to
derive suitable characterizations and Sobolev and Moser inequalities
for the twisted Sobolev spaces that we use in this paper. These are
not direct consequences of the standard proofs, while Hardy-type
inequalities play an important role throughout. In
Section~\ref{S.elliptic} we prove estimates at infinity for the
elliptic part of the wave equation in an asymptotically AdS space (see
Theorem~\ref{T.ellip} and Corollary~\ref{C.ellip}). Analogous
estimates for asymptotically Minkowskian space-times, where the time
slices are asymptotically Euclidean, were derived by Christodoulou and
Choquet-Bruhat in~\cite{CC}.  The elliptic estimates are used in
Section~\ref{WaveProp} to derive energy estimates for the Cauchy
problem for the wave equation in an asymptotically AdS patch (see
Theorems~\ref{P.H1AdS} and~\ref{T.higherwave}). To deal with the data
on the conformal boundary, one uses an additional set of results
obtained in Section~\ref{S.boundary}, where the layers of the solution
that are large at infinity are ``peeled off'' (see
Theorem~\ref{T.large}) and the remaining part of the solution is
carefully controlled. In Section~\ref{S.LW} we finally discuss the
global structure of asymptotically AdS space-times and prove the global
well-posedness of the problem for the linear wave
equation~\eqref{LW}. The local well-posedness of the problem for the
nonlinear wave equation~\eqref{NLW} with the
nonlinearity~\eqref{nonlinF} is finally proved in Section~\ref{S.NLW}
using an iterative argument.

\section{Definitions and notation}\label{DefNot}
\label{S.defs}

In order to describe the geometry of an asymptotically AdS
space-time at infinity (i.e., in a neighborhood of an end), in this section
we define the notion of an asymptotically AdS patch. Most of our work
in this paper will take place in asymptotically AdS patches. In this
section we also introduce twisted Sobolev spaces which are adapted by
definition to the boundary behavior at infinity of the metric in an
asymptotically AdS patch. We shall see in Sections~\ref{Ell} and ~\ref{WaveProp} that these function spaces are well adapted to the study of the the elliptic estimates at infinity and the energy estimates for wave propagation which are needed to prove the well-posedness of the holographic prescription problem.  

We begin by introducing some basic notation to describe the asymptotics of functions. Consider the manifold $\RR\times (0,1)\times \SS^{n-2}$ with coordinates $(t,x,\te)$. We shall commit a slight abuse of notation and think if the ``coordinate'' $\te$ as taking 
values in $\SS^{n-2}$. Here we say that some quantity $Q(t,x,\te)$ is
of order $\cO(x^m)$ if there exist constants $ C_{jkl}$ such that
\[
\big| \pd_t^j\pd_x^kD_\te^lQ(t,x,\te)\big|\leq C_{jkl}\, x^{m-k}
\]
for $x$ close to $0$, uniformly for all
$(t,\te)\in\RR\times\SS^{n-2}$. Notice that, as is customary when
considering functions on a manifold, the angular derivatives $D_\te^l u$
must be interpreted either using local coordinates in the obvious way
or, more intrinsically, using tangent vector fields, but we will
omit this point whenever we find it notationally convenient. A similar
abuse of notation will be often made when dealing with Sobolev spaces
as in Eq.~\eqref{bH} below.

With this notation in place, we are now ready to define the concept of
an asympotically AdS patch. For convenience, we will include in the definition a small
parameter $a$ that describes the width of the patch. 

\begin{definition}\label{D.aAdS}
  As {\em asymptotically AdS patch}\/ (of width $a$) in a Lorentzian manifold $\cM$ is an open set $\cU\subset \cM$
  with smooth boundary that is covered by coordinates
\[
(t,x,\te)\in\RR\times (0,a)\times \SS^{n-2}
\]
in which the coefficients of the metric are smooth and read as
\begin{gather*}
g_{tt}=-x^{-2}-1+\cO(x^2)\,,\qquad g_{xx}=x^{-2}-1+\cO(x)\,,\qquad
g_{tx}=\cO(x)\,,\\
g_{\te^i\te^j}=x^{-2}(g_{\SS^{n-2}})_{ij}+\cO(x)\,, \qquad g_{t\te^i}=\cO(x)\,,\qquad g_{x\te^i}=\cO(x^2)\,,
\end{gather*}
where $g_{\SS^{n-2}}$ stands for the metric of the unit $(n-2)$-sphere. 
\end{definition}

In an asymptotically AdS patch, the wave equation~\eqref{LW}
can be written, after dividing by the coefficient of
$\phi_{tt}$ (which is of the form $x^2+\cO(x^4)$), in the form
\begin{multline}\label{largo}
  -\pd_{t}^2\phi+\pd_{x}^2\phi-\frac{n-2}x\pd_x\phi
  +\De_\te\phi-\frac\mu{x^2}\phi= \cO(x)\, D_{tx\te}\phi +\cO(x^2)\,
  D_{x\te}D_{tx\te}\phi\,,
\end{multline}
where the symbols $D_{x\te}$ and $D_{tx\te}$ stand for the 
derivatives of~$\phi$ with respect to all the space and space-time
variables, respectively. 

Following~\cite{Warnick}, it will be convenient to introduce the
function
\begin{equation}\label{phiu}
u:=x^{\frac{1-n}2}\phi\,,
\end{equation}
in terms of which Eq.~\eqref{largo} becomes
\begin{equation}\label{Pg}
P_gu=0\,,
\end{equation}
where $P_g$ is a second-order differential operator of the form
\begin{equation}\label{eq-u}
P_gu:=-\pd_{t}^2u+\pd_{x}^2u+\frac1x\pd_x u
+\De_\te u-\frac{\al^2}{x^2} u+ \cO(x)\, D_{tx\te}u +\cO(x^2)\,
  D_{x\te}D_{tx\te}u\,.
\end{equation}
Let us recall that the constant $\al$ was introduced in
Eq.~\eqref{al} and record for future reference that the precise
relationship between $\square_g\phi$ and $P_gu$ is
\begin{equation}\label{BoxPg}
\square_g\phi-\mu\phi=\big[1+\cO(x^2)\big] x^{\frac{n+3}2}\, P_g u\,.
\end{equation}
A key observation, due to Warnick~\cite{Warnick}, is
that the terms of $P_gu$ that are dominant for small $x$ can be rewritten as
\begin{equation*}
-\pd_{t}^2u+\pd_{x}^2u+\frac{\pd_x u}x
+\De_\te u-\frac{\al^2u}{x^2}=-\pd_{t}^2u- \bD_x^*\bD_xu
+\De_\te u\,,
\end{equation*}
where the {\em twisted derivative}
\begin{equation}\label{twist}
\bD_xu:=x^{-\al}\pd_x\big(x^\al u\big)=u_x+\frac\al xu\,,
\end{equation}
is directly related to the ``renormalized energy'' considered by
Breitenlohner and Freedman~\cite{BF} and 
\[
\bD_x^*u:=-x^{\al-1}\pd_x\big(x^{1-\al} u\big)=-u_{x}+\frac{\al-1} xu\,,
\]
is the formal adjoint of $\bD_x$ with respect to the scalar product of
the space
\begin{equation}\label{weightL2}
\bL^2:= L^2\big((0,a)\times\SS^{n-2},x\, dx\, d\te\big)\,,
\end{equation}
where $d\te$ denotes the standard measure on the unit sphere.
It should be noticed that, on account of the relationship between $\phi$ and $u$
and of the coordinate expression for the metric in an asymptotically AdS
patch, the $L^2$ norm of $\phi$ with respect to the natural space-time measure,
\[
d\mathrm{vol}_{\cM}:=\sqrt{|\det g|}\, dx\, d\te^1 \cdots
d\te^{n-2}\, dt\,,
\]
is equivalent to the $L^2_t\bL^2$ norm of $u$ for small $a$, since
\begin{align*}
\int_{\cU}\phi^2\, d\mathrm{vol}_{\cM}&=\int_{\RR\times(0,a)\times\SS^{n-2}}\big(1+\cO(x)\big)\, u^2\, x\, dx\,
d\te\, dt\\ 
&=\big(1+\cO(a)\big)\int_{-\infty}^{\infty}\|u(t,\cdot)\|_{\bL^2}^2\, dt\,.
\end{align*}

To control functions defined on an asymptotically AdS patch, we will consider Sobolev spaces associated with the twisted derivatives, which, setting
$\bH^0\equiv \bL^2$, can be recursively defined as
\begin{subequations}\label{bH}
\begin{align}
\bH^{2j+1}&:=\big\{ v\in \bH^{2j}: D_\te v\in
\bH^{2j}\,,\; \bD_x(\bD_x^*\bD_x)^jv\in \bL^2\big\}\,,\\
\bH^{2j+2}&:=\big\{ v\in \bH^{2j+1}: D_\te v\in
\bH^{2j+1}\,,\; (\bD_x^*\bD_x)^{j+1}v\in \bL^2\big\}\,,
\end{align}
\end{subequations}
with $j\geq0$. Notice that the functions in $\bH^k$ are defined on
$(0,a)\times\SS^{n-2}$, which corresponds to the spatial part of an
asymptotically AdS patch $\cU$. The norm associated to $\bH^k$ (resp.\
the Lebesgue space $\bL^2$) will be denoted by $\|\cdot\|_{\bH^k}$
(resp.\ $\|\cdot\|_{\bL^2}$).

It should be noted that the $\bH^{2}$ twisted Sobolev spaces were first introduced and studied in the case $\alpha<1$ by Warnick~\cite{Warnick}.

We will also need Sobolev spaces corresponding to Dirichlet boundary
conditions. We will denote by $\bH_0^1$ the twisted Sobolev space on
$(0,a)\times\SS^{n-2}$ with zero trace on the inner and outer
boundaries $\{x=0\}\cup\{x=a\}$. We shall not elaborate on the
properties of the trace map here, since in the following section we
will give a more direct characterization of these spaces (see
Proposition~\ref{P.vDv} and Corollary~\ref{C.density}).

\section{Inequalities for twisted Sobolev spaces}\label{SobIneq}
\label{S.Sobolev}

Our goal in this section is to derive inequalities for the twisted
Sobolev spaces defined above using some integral operators $A$ and $A^*$ which act as inverses of the twisted derivatives $\bD_x$ and $\bD_x^*$.  This will be done via Hardy-type inequalities, which we now derive.

Consider the integral operator $A$ defined by
\[
A\vp(x):=x^{-\al}\int_0^x y^\al \vp(y)\, dy\,.
\]
An immediate observation is that $A$ is a right inverse of the twisted
derivative $\bD_x$, that is,
\[
\bD_x(A\vp)=\vp\,.
\]
We will be concerned with the mapping properties of $A$ for several
weighted Lebesgue spaces.  A special role will be played by the space
\[
\bL^2_x:=  L^2\big((0,a),x\, dx\big)\,.
\]
The reason for which we will be interested in this space is that we
have an obvious decomposition of $\bL^2$ of the form 
\[
\bL^2=\bL^2_xL^2_\te\,,
\]
where $L^2_\te:= L^2(\SS^{n-2})$ is the usual $L^2$ space on the
sphere. Therefore, the space $\bL^2_x$ can be interpreted to some
extent as the ``radial'' part of the space of square-integrable
functions in (the spacial part of) an asymptotically AdS patch.

The adjoint of $A$ with respect to the inner product of $\bL^2_x$ is given by
\[
A^*\vp(x):=x^{\al-1}\int_x^a y^{1-\al}\vp(y)\, dy\,.
\]
One can easily check that $A^*$ is a right inverse of $\bD_x^*$, that is,
\[
\bD_x^*(A^*\vp)=\vp\,.
\]

In the next theorem, we derive estimates for $A\vp$ and $A^*\vp$ assuming that we have 
$\bL^2_x$ bounds on $\vp$. In some results that we will prove later
on, it will be convenient to make explicit the dependence on the parameter $a$
of a few upper bounds. Therefore, here and in what follows we will use the
capital letter $C$ to denote a positive constant that can depend on
the parameter~$a$ and the lower case $c$ to denote a positive,
$a$-independent constant. As customary, both constants may vary from
line to line.

\begin{theorem}\label{T.estA}
For any reals $s<\al$ and $r\geq s$, the operators $A,A^*$ have the
following properties:
\begin{subequations}
\begin{align}
\|x^{r-1}A\vp\|_{\bL^2_x}&\leq ca^{r-s}\|x^s\vp\|_{\bL^2_x}\,,\label{AL2}\\[1mm]
\|A\vp\|_{L^\infty_x}&\leq c\|\vp\|_{\bL^2_x}\,,\label{ALinf}\\[1mm]
\|x^{-s}A^*\vp\|_{\bL^2_x}&\leq
ca^{r-s}\|x^{1-r}\vp\|_{\bL^2_x}\,,\label{A*L2}\\[1mm]
|A^*\vp(x)|&\leq c\|\vp\|_{\bL^2_x}\times\begin{cases}
1 &\text{if }\al>1\,,\\
\log(a/x) &\text{if }\al=1\,,\\
(a/x)^{1-\al}&\text{if }\al<1\,.
\end{cases}\label{A*Linf}
\end{align}
\end{subequations}
Notice, in particular, that for $\al>1$ one
has $\|A^*\vp\|_{L^\infty_x}\leq c \|\vp\|_{\bL^2_x}$.
\end{theorem}

\begin{proof}
  Let us begin by proving the inequality~\eqref{A*L2}, since the fact
  that $A^*\vp(a)=0$ simplifies the integration by parts. In view of
  the formula for $A^*$, we need the Hardy inequality
\[
\int_0^ax^{2\al-2s-1}\bigg(\int_x^ay^{1-\al}\vp(y)\,
dy\bigg)^2dx\leq c a^{2r-2s}\int_0^a x^{3-2r}\vp(x)^2\, dx\,.
\]
To prove this, let us call $J^2$ the LHS of this inequality and set
\[
\psi(x):=\int_x^ay^{1-\al}\vp(y)\,dy\,.
\]
Then integrating by parts and using the Cauchy--Schwarz inequality we
find
\begin{align*}
J^2&=\int_0^ax^{2\al-2s-1}\psi^2\, dx
=\frac1{\al-s}\int_0^a \psi\,\vp\, x^{\al-2s+1}\,
dx\\
&=\frac1{\al-s}\int_0^a x^{r-s} \, (x^{\al-s-\frac12}\psi)\,
(x^{\frac32-r}\vp)\, dx\\
&\leq \frac{a^{r-s}}{\al-s}\, J\, \bigg(\int_0^a x^{3l-2r}\vp^2\, dx\bigg)^{1/2}\,,
\end{align*}
where in the second and fourth lines we have used that $s<\al$ and $r\geq
s$, respectively. This inequality is~\eqref{A*L2}.

Since this implies that
\[
A^*: L^2(x^{2-2r}\, x\, dx)\to L^2(x^{-2s}\, x\, dx)
\]
is bounded, by duality it stems that the adjoint with respect to
$\bL^2_x=L^2(x\, dx)$ is a bounded map
\[
A: L^2(x^{2s}\, x\, dx)\to L^2(x^{2r-2}\, x\, dx)
\]
with the same norm, which immediately yields~\eqref{AL2}.

Let us now pass to the pointwise bounds. To prove~\eqref{A*Linf} we
utilize the Cauchy-Schwarz inequality to write (for $\al\neq1$)
\begin{align*}
  \big|A^*\vp(x)\big|&=x^{\al-1}\bigg|\int_x^ay^{1-\al}\vp(y)\,
  dy\bigg|\\
  & \leq x^{\al-1}\bigg(\int_x^ay^{1-2\al}\,
  dy\bigg)^{1/2}\bigg(\int_x^ay\, \vp(y)^2\, dy\bigg)^{1/2}\\
&\leq \|\vp\|_{\bL^2_x}\bigg[\frac1{2-2\al}\bigg(\bigg(\frac ax\bigg)^{2-2\al}-1\bigg)\bigg]^{1/2}\,.
\end{align*}
The estimate~\eqref{A*Linf} follows from this
inequality and the elementary observation
\[
\bigg|\bigg(\frac ax\bigg)^{2-2\al}-1\bigg|^{1/2}\leq \begin{cases}
1 &\text{if } \al>1\,,\\
(a/x)^{1-\al}  &\text{if } \al<1\,.
\end{cases}
\]
The case $\al=1$ and the inequality~\eqref{ALinf} follow from an
analogous argument.
\end{proof}

In the following proposition we provide pointwise bounds for $A\vp$
and $A^*\vp$ under the assumption that we have pointwise bounds for
$\vp$. Our goal here is to relate the fall off of the former at $x=0$ to that
of the latter using power laws. We shall use the notation $x\wedge y:=\min(x,y)$.

\begin{proposition}\label{P.Apoint}
Let $\vp,\psi$ be functions satisfying
\[
|\vp(x)|\leq C_1x^s\quad\text{and} \quad |\psi(x)|\leq C_2x^{r}\,,
\]
with $s>-1-\al$ and an arbitrary real $r$. Then $A\vp$ and $A^*\vp$  obey the pointwise
bounds
\begin{align*}
\big|A\vp(x)\big|&\leq CC_1x^{s+1}\,,\\
\big|A^*\psi(x)\big|&\leq C C_2 x^{(r+1)\wedge (\al-1)}\,,
\end{align*}
where $C$ does not depend on the functions $\vp$ or $\psi$.
\end{proposition}

\begin{proof}
It is easy to find that
\[
\big|A\vp(x)\big|\leq x^{-\al}\int_0^xC_1y^{\al+s}\, dy=CC_1x^{s+1}\,.
\]
Likewise,
\begin{align*}
\big|A^*\psi(x)\big|&\leq C_2x^{\al-1}\int_x^ay^{r-\al+1}\,
dy\\
&=CC_2x^{\al-1}\big(a^{r-\al+2}-x^{r-\al+2}\big)\leq CC_2x^{(r+1)\wedge(\al-1)}\,.
\end{align*}
\end{proof}


In the rest of this section we shall see how the previous estimates
for the operators $A$ and $A^*$ are of use in the analysis of the
Sobolev spaces $\bH^k$. Hence we now focus our attention on the action of the operators $A,A^*$ not only on
functions of the variable $x$, but also on functions that depend on
$x$ and $\te$. An easy result that we will need later on
is the following. We recall that we are using the
shorthand notation $L^2_\te$ for the space of square-integrable functions on $\SS^{n-2}$.

\begin{proposition}\label{P.vDv}
Let $v(x,\te)$ be a function in $\bL^2$ such that $\bD_x v$ is also is
$\bL^2$. Then 
\[
v=\begin{cases}
A(\bD_xv) &\text{if }\al\geq 1\,,\\
A(\bD_xv)+x^{-\al}\ms G(\te) &\text{if }\al<1\,.
\end{cases}
\]
Here $G(\te)$ is some function in $L^2_\te$, which corresponds
to the trace of the function $x^\al v$ on the set $\{x=0\}$ and
satisfies
\[
\|G\|_{L^2_\te}\leq C\big(\|v\|_{\bL^2}+\|\bD_xv\|_{\bL^2}\big)\,.
\]
If $w(x,\te)$ is a function in $\bL^2$ such that $\bD_x^*w\in\bL^2$,
then
\[
w=A^*(\bD_xw)+x^{\al-1}\ms H(\te)
\]
for some function $H(\te)\in L^2_\te$ with
\[
\|H\|_{L^2_\te}\leq C\big(\|w\|_{\bL^2}+\|\bD_x^*w\|_{\bL^2}\big)\,.
\]
\end{proposition}
\begin{proof}
Since $A$ is a right inverse of $\bD_x$ and $A$ maps $\bL^2$ to itself
by Proposition~\ref{T.estA}, the expression for $v$
follows from the fact that the only functions in the kernel of $\bD_x$
are those of the form $x^{-\al}G(\te)$, and that they are in $\bL^2$
if and only if $\al<1$. The estimate for the norm of $G$ follows from
\begin{align*}
\| x^{-\al}G\|_{\bL^2}&=\frac{a^{1-\al}}{\sqrt{2-2\al}}
\|G\|_{L^2_\te}\\
&\leq \|v \|_{\bL^2}+\|A(\bD_xv) \|_{\bL^2} \\
&\leq
\|v \|_{\bL^2}+ C\|\bD_xv \|_{\bL^2} \,,
\end{align*}
where we have used~\eqref{AL2} to estimate the norm of $A(\bD_x
v)$. By construction, $G(\te)$ is the trace of
$x^\al \,v(x,\te)$ on $\{x=0\}$. The proof of the formula for $w$
follows the same lines.
\end{proof}

\begin{remark}
  Notice that, for $\al<1$, a function $v\in\bH^1$ with zero trace on $\{x=a\}$ is in $\bH^1_0$ if
  and only if the function $G(\te)$ considered in
  Proposition~\ref{P.vDv} is zero, that is, if $v=A(\bD_xv)$. When $v\in\bH^1$ one can show
  that the trace $G(\te)$ to $\{x=0\}$ actually belongs to $H^{\al}(\SS^{n-2})$,
  but we will not need this fact.
\end{remark}

Since $C^\infty\cpt((0,a)\times\SS^{n-2})$ is clearly dense in
$\bL^2$, a corollary of the previous proposition is the following
characterization  of the twisted Sobolev space $\bH^1_0$ with zero
trace. We recall that a function is differentiable in a closed
interval if it differentiable in a open interval containing it.

\begin{corollary}\label{C.density}
  The Sobolev space $\bH_0^1$ is the completion in the $\bH^1$-norm of the space
  $C^\infty\cpt((0,a)\times\SS^{n-2})$ and any $v\in\bH^1_0$ satisfies
\begin{equation}\label{L2H1}
\|v\|_{\bL^2}\leq ca\|\bD_xv\|_{\bL^2}\,,
\end{equation}
where the constant $c$ is independent of $a$. For $\al<1$, the space
  $x^{-\al}\, C^\infty\cpt([0,a)\times\SS^{n-2})$ is dense in the space of
  $\bH^1$ functions with zero trace on $\{x=a\}$. 
\end{corollary}
\begin{proof}
  The bound of the $\bL^2$ norm of $v$ in terms of that of $\bD_xv$ is
  an immediate consequence of the expression $v=A(\bD_x v)$ and the
  estimates for $A$ proved in Theorem~\ref{T.estA}. Since
  $C^\infty\cpt((0,a)\times\SS^{n-2})$ is clearly dense in $\bL^2$,
  there is a smooth, compactly supported function $F(x,\theta)$ that
  approximates $\bD_xu$ in the $\bL^2$ norm. By the estimates proved
  in Theorem~\ref{T.estA} and the expression of $A$, it then follows
  that $A(F)$ is in $C^\infty\cpt((0,a]\times\SS^{n-2})$ and
  approximates $A(\bD_xu)$ in the $\bL^2$ norm.

  By Proposition~\ref{P.vDv}, for $\al\geq1$ this proves that
  $C^\infty\cpt((0,a]\times\SS^{n-2})$ is dense in the space of
  $\bH^1$ functions whose trace to $\{x=0\}$ is zero. Since the
  $\bH^1$-norm is equivalent to the standard $H^1$-norm in a
  neighborhood of the other component of the boundary, $\{x=a\}$, it
  then follows that $C^\infty\cpt((0,a)\times\SS^{n-2})$ is dense in
  $\bH^1_0$ for $\al\geq1$.  For $\al<1$, it suffices to write
\[
u=A(\bD_xu)+x^{-\al}G(\te)
\]
and combine the discussion of $A(\bD_xu)$ with the observation that
the function $G(\te)$ can be approximated by a smooth function $\tilde
G\in C^\infty(\SS^{n-2})$.
\end{proof}

It should be noted that in the case $\alpha<1$, Proposition~\ref{P.vDv} and Corollary~\ref{C.density} were proved by Warnick in~\cite{Warnick}.

To pass from estimates in twisted Sobolev spaces to pointwise estimates, the key is the
following Morrey-type inequality for $\bH^m$. It is worth noticing
that, although the functions in this space are defined on an
$(n-1)$-dimensional manifold, we will only get pointwise estimate for
$m<\frac n2$, instead of for $m<\frac{n-1}2$ as one would in standard
Sobolev spaces $H^m(\RR^{n-1})$. This can be understood to be a
consequence of the singular behavior of the measure and of the twist
factor at the conformal boundary. Likewise, we are interested in controlling the growth of functions
near the conformal boundary. Before stating the result, it is convenient to introduce some
notation for ordered twisted derivatives and write
\begin{equation}\label{D(m)}
\bD_x^{(m)}v:=\begin{cases}
(\bD_x^*\bD_x)^\frac m2v & \text{if } m \text{ is even},\\[1mm]
\bD_x (\bD_x^*\bD_x)^{\frac{m-1}2}v & \text{if } m \text{ is odd}.
\end{cases}
\end{equation}
This will be used throughout this paper.

\begin{theorem}\label{T.Morrey}
Let $v\in\bH^m$ with $m>\frac n2$. Then for any non-negative integers
$i,j$ with
\[
i+j\leq m-\frac n2\,,
\]
the following inequalities hold:
\begin{enumerate}
\item If $\al>1$,
\[
\big|D_\te^i\bD_x^{(j)} v(x,\te)\big|\leq C\, \|v\|_{\bH^m}\, x^{(m-j-1)\wedge \al} \,.
\]

\item If $\al<1$,
\[
\big|D_\te^i\bD_x^{(j)} v(x,\te)\big|\leq C\,
\|v\|_{\bH^m}\,  \begin{cases}
x^{-\al} & \text{if } j \text{ is even},\\
x^{\al-1} & \text{if } j \text{ is odd}.
\end{cases}
\]

\item If $\al=1$, the estimate in~(i) still holds after replacing $m$ by $m-\de$ in the
  exponent, with any $\de>0$.
\end{enumerate}
\end{theorem}

\begin{remark}\label{R.uniform}
When $v\in \bH^1_xH^{\si}_\te\cap \bH^1_0$ with $\si>s+\frac n2-1$ for some
nonnegative integer~$s$, the estimate for $v$ replaced by the uniform bound
\[
\|D_\te^l v \|_{L^\infty}\leq C_{\si}\,
\|v\|_{\bH^1_xH^{\si}_\te}
\]
for $0\leq l\leq s$.
\end{remark}

By the Sobolev embedding theorem, Theorem~\ref{T.Morrey} and the
subsequent remark are a straightforward
consequence of the following

\begin{lemma}\label{L.Morrey}
Let $v\in\bH^m_xL^\infty_\te$, with $m\geq1$. Then, for any integer
$j\leq m-1$, the following 
inequalities hold:
\begin{enumerate}
\item If $\al>1$,
\[
\big|\bD_x^{(j)} v(x,\te)\big|\leq C\, \|v\|_{\bH^m_xL^\infty_\te}\, x^{(m-j-1)\wedge \al} \,.
\]

\item If $\al<1$,
\[
\big|\bD_x^{(j)} v(x,\te)\big|\leq C\,
\|v\|_{\bH^m_xL^\infty_\te}\, \begin{cases}
x^{-\al} & \text{if } j \text{ is even},\\
x^{\al-1} & \text{if } j \text{ is odd}.
\end{cases}
\]
If $v\in \bH^1_xL^\infty_\te\cap \bH^1_0$, the estimate for $v$ can be
refined to
\[
\|v \|_{L^\infty}\leq C\,
\|v\|_{\bH^1_xL^\infty_\te}\,.
\]

\item If $\al=1$, the estimate~(i) still holds after replacing $m$ by $m-\de$ in the
  exponent, with any $\de>0$.
\end{enumerate}
\end{lemma}
\begin{proof}
In the proof of this theorem we will
use the notation
\[
v_{j}:=\bD_x^{(j)} v\,.
\]
To keep things concrete, we will see how the result is proved for
small values of $m$.

When $m=1$, we have that $v_1\in \bL^2_x L^\infty_\te$, so it follows
from Theorem~\ref{T.estA} that
\[
|Av_1(x,\te)|\leq\sup_\te\|A v_1(\cdot,\te)\|_{L^\infty_x}\leq C \sup_\te
\|v_1(\cdot,\te)\|_{\bL^2_x}=C\|v_1\|_{\bL^2_xL^\infty_\te}\leq C\|v\|_{\bH^1_xL^\infty_\te}\,,
\]
where $\sup$ stands for the essential supremum.
Proposition~\ref{P.vDv} ensures that $v=Av_1+x^{-\al}G_0(\te)$, the
second summand being absent if $\al\geq1$ or $v\in\bH^1_0$. When these conditions are not satisfied, we can
estimate the second summand by noticing that
\begin{align*}
\|G_0\|_{L^\infty_\te}&=
C\|x^{-\al}G_0(\te)\|_{\bL^2_xL^\infty_\te}\leq
C\|v\|_{\bL^2_xL^\infty_\te}+ C\|Av_1\|_{\bL^2_xL^\infty_\te}\\
&\leq C\|v\|_{\bL^2_xL^\infty_\te}+ C\|v_1\|_{\bL^2_xL^\infty_\te}\leq C\|v\|_{\bH^1_xL^\infty_\te}
\end{align*}
using again Theorem~\ref{T.estA}. The estimate for $v$ then follows.

When $m=2$, we have an $\bL^2L^\infty$ bound on $v_2$, so
Theorem~\ref{T.estA} ensures that 
\[
|A^*v_2(x,\te)|\leq C\|v\|_{\bH^2_xL^\infty_\te}x^{-\be}\,,
\]
where 
\[
\be:=\begin{cases}
0 &\text{if }\al>1\,,\\
\de &\text{if }\al=1\,,\\
1-\al &\text{if }\al=1\,.
\end{cases}
\]
Here $\de$ is an arbitrarily small positive constant that we introduce
to take care of the logarithmic term in
Theorem~\ref{T.estA}. Proposition~\ref{P.vDv} then gives
\[
v_1=A^*v_2+ x^{\al-1}G_1(\te)\,,
\]
so with the above bound for $A^*v_2$ and using that
\begin{align*}
\|G_1\|_{L^\infty_\te}&=
C\|x^{\al-1}G_1(\te)\|_{\bL^2_xL^\infty_\te}\leq
C\|v_1\|_{\bL^2_xL^\infty_\te}+ C\|A^*v_2\|_{\bL^2_xL^\infty_\te}\\
&\leq C\|v\|_{\bH^2_xL^\infty_\te}
\end{align*}
we find
\[
|v_1(x,\te)|\leq C\|v\|_{\bH^2_xL^\infty_\te}\, x^{(\al-1)\wedge (-\be)}\,.
\]
We can use this
pointwise estimate for $v_1$ and Proposition~\ref{P.Apoint}
to infer that
\[
\big|Av_1(x,\te)\big|\leq C\|v\|_{\bH^2_xL^\infty_\te} \, x^{\al\wedge (1-\be)}
\]
Since
\[
v=Av_1+x^{-\al} G_0(\te)
\]
and $G_0$ can be controlled as in the case $m=1$ to find that it
does not appear for $\al\geq1$ or $v\in \bH^1_0$ and
satisfies $\|G_0\|_{L^\infty_\te}\leq C\|v\|_{\bH^2_xL^\infty_\te}$
otherwise, we arrive at the desired estimate in the case $m=2$.

Now that we have established the cases $m=1$ and $m=2$, the general
case follows from a totally analogous reasoning and a simple induction
argument.
\end{proof}

Later on in the paper we will need to estimate the norms of products
of functions. This will be accomplished using the following result,
which is a Moser estimate for the twisted Sobolev space
$\bH^m$. Just as in Theorem~\ref{T.Morrey}, we will require functions
in $\bH^m$ with $m>\frac n2$, rather than $m<\frac{n-1}2$ as in the
Euclidean case.

\begin{proposition}\label{P.multilinear}
Given $m>\frac n2$, let us consider nonnegative integers $j_1,\dots,
j_l$ and $k_1, \dots, k_l$ with total
sum
\[
j_1+\cdots+j_l+k_1+\cdots+k_l\leq m\,.
\]
Furthermore, let us set
\[
\eta:=\begin{cases}
0& \text{if }\al>1\,,\\
\al\wedge(1-\al)& \text{if }\al>1\,.
\end{cases}
\]
Then the inequality
\[
\big\| x^{(l-1)\eta} \bD_x^{(k_1)}D_\te^{j_1}u_1\cdots
\bD_x^{(k_l)}D_\te^{j_l}u_l\big\|_{\bL^2}\leq C \|u_1\|_{\bH^m}\cdots\|u_l\|_{\bH^m}
\]
holds for any functions $u_1,\dots, u_l$ in $\bH^m$. For $\al=1$,
the result is still true for any positive (but arbitrarily small) $\eta$.
\end{proposition}
\begin{proof}
  Let us prove the statement for $l=2$, which is the case that will be
  needed in this paper. The general result follows from an analogous
  argument using the generalized Cauchy--Schwarz inequality. 

  Since
\[
\|x^\eta u_1\|_{L^\infty}\leq C\|u_1\|_{\bH^m}
\]
by Theorem~\ref{T.Morrey}, we have
\begin{align}
\|x^\eta u_1\, \bD_x^{(k_2)}D_\te^{j_2}u_2\|_{\bL^2} &\leq C\|x^\de u_1\|_{L^\infty}\,
\|\bD_x^{(k_2)}D_\te^{j_2}u_2\|_{\bL^2} \notag \\
&\leq C\|u_1\|_{\bH^m}\,
\|\bD_x^{(k_2)}D_\te^{j_2}u_2\|_{\bL^2}\notag \\
&\leq C \|u_1\|_{\bH^m}\|u_2\|_{\bH^m}\,,\label{multieasy}
\end{align}
which proves the result  when $j_1$ and $k_1$ are zero. Hence, by symmetry we can henceforth assume that both
$j_1+k_1$ and $j_2+k_2$ are nonzero.

It follows from Proposition~\ref{P.vDv} that 
\begin{equation}\label{multi1}
\bD_x^{(k_1)}D_\te^{j_1}u_1= A^{\#}(\bD_x^{(k_1+1)}D_\te^{j_1}u_1)+ x^{-\eta_1}\,F(\te)\,,
\end{equation}
where $A^{\#}$ (resp.\ $\eta_1$) stands for $A$ or $A^*$ (resp.\
$-\al$ or $\al-1$) depending on the parity of
$k_1$. The terms with exponent $\eta_1=-\al$ do not appear if
$\al\geq1$. Taking $s$ derivatives with respect to $\te$ in the
identity~\eqref{multi1} and using the estimates for the operators
$A$ and $A^*$ proved in Theorem~\ref{T.estA}, we immediately find that
\begin{align*}
\big\| x^\de \bD_x^{(k_1)}
D_\te^{j_1+s}u_1\big\|_{L^\infty_xL^2_\te}&\leq C\big( \| \bD_x^{(k_1)}
D_\te^{j_1+s}u_1\|_{\bL^2} + \| \bD_x^{(k_1+1)}
D_\te^{j_1+s}u_1\|_{\bL^2}\big)\\
&\leq C\|u_1\|_{\bH^{j_1+k_1+s+1}}\,.
\end{align*}
Since $j_1+k_1$ is now at most $m-1$, by the Sobolev embedding theorem this yields
\[
\|x^\eta \bD_x^{(k_1)}
D_\te^{j_1}u_1\big\|_{L^\infty_xL^{p_1}_\te}\leq C\|x^\eta \bD_x^{(k_1)}
D_\te^{j_1}u_1\big\|_{L^\infty_x H^{m-j_1-k_1-1}_\te}\leq C\|u_1\|_{\bH^m}\,,
\]
where the exponent $p_1\in[2,\infty]$ is $\infty$ if $m-k_1-j_1>\frac
n2$ and the reciprocal of 
\[
\frac12-\frac{m-k_1-j_1-1}{n-2}
\]
if $m-k_1-j_1<\frac n2$. In the limiting case, $m-k_1-j_1=\frac n2$, we can
take any finite value of $p_1$. 

In the first case, namely, $m-k_1-j_1>\frac n2$, we then have an
$L^\infty$ bound for $x^\eta \bD_x^{(k_1)}
D_\te^{j_1}u_1$, so the estimate follows just as
in~\eqref{multieasy}, reversing the roles of $u_1$ and $u_2$. Hence
let us assume that we are not in this case and notice that the inequality
\[
\|\bD_x^{(k_2)} D_\te^{j_2}u_2\|_{\bL^2_xL^{p_2}_\te}\leq C
\|\bD_x^{(k_2)} D_\te^{j_2}u_2\|_{\bL^2_xH^{m-j_2-k_2}_\te}\leq C\|u_2\|_{\bH^m}
\]
holds provided that the exponent $p_2$ is chosen as
\[
p_2:=\begin{cases}
\infty &\text{if } m-k_2-j_2> \tfrac n2\,,\\[1mm]
(\frac12-\frac{m-k_2-j_2}{n-2})^{-1} &\text{if } m-k_2-j_2< \tfrac n2\,.
\end{cases}
\]
When $m-k_2-j_2= \tfrac n2$, one can take any finite $p_2$. Since
\[
j_1+j_2+k_1+k_2\leq m \quad \text{and}\quad m>\frac n2\,,
\]
a straightforward computation shows that
\[
\frac1{p_1}+\frac1{p_2}<\frac12\,.
\]
(In the limiting case, of course, one has to choose the exponent
large enough.) Hence one can use the Cauchy--Schwarz inequality to
arrive at
\begin{align*}
  \|x^\eta \bD_x^{(k_1)}D_\te^{j_1}u_1\,
  \bD_x^{(k_2)}D_\te^{j_2}u_2\|_{\bL^2}&= \bigg(\int  x^{2\eta} (\bD_x^{(k_1)}D_\te^{j_1}u_1)^2\,
  (\bD_x^{(k_2)}D_\te^{j_2}u_2)^2\, x\, dx\, d\te\bigg)^{\frac12}\\
&\leq   \|x^\eta \bD_x^{(k_1)}D_\te^{j_1}u_1\|_{L^\infty_xL^{p_1}_\te}
\, \| \bD_x^{(k_2)}D_\te^{j_2}u_2\|_{\bL^2_xL^{p_2}_\te}\\
&\leq C\|u_1\|_{\bH^m}\,\|u_2\|_{\bH^m}\,,
\end{align*}
as claimed.
\end{proof}

For completeness, we shall conclude this section with a compactness
result for twisted Sobolev spaces. The proof we give
follows~\cite{HW}, where this was proved for $\al<1$ in a
significantly more general setting.

\begin{proposition}\label{P.compactness}
Let $u_j$ be a sequence bounded in $\bH^1$ that converges weakly to
some $u\in \bH^1$. Then $u_j$ also converges strongly, that is, $\|u_j-u\|_{\bL^2}\to0$.
\end{proposition}

\begin{proof}
To estimate the $\bL^2$-norm of $u_j-u_k$, let us start by dividing
the set $(0,a)\times\SS^{n-2}$ in convex regions of the form 
\[
V:=(x_0,x_0+h)\times B\,,
\]
where $h$ is a small positive number and $B$ is a small set in $\SS^{n-2}$. Without loss of generality,
we can take Cartesian coordinates in $B$ (which, with some abuse of
notation, will be denoted by $\te=(\te_1,\dots,\te_{n-2})$) and
assume that $B$ is the set $\{\max_j|\te_j|<h\}$. It is clear that
we can cover $(0,a)\times\SS^{n-2}$ by a finite collection $\cV_h$ of sets as above, so that 
\begin{equation}\label{cV}
 \| w \|_{\bL^2}^2 \leq \sum_{V\in\cV_h}  \|
w \|_{\bL^2(V)}^2\leq C  \| w \|_{\bL^2}^2 
\end{equation}
for some $C$ independent of $h$ and all $w\in\bL^2$. We will
eventually apply this inequality to the difference $w:=u_j-u_k$, where
$j$ and $k$ are large integers.

To estimate the norm of a function on the set $V\in\cV_h$ we
will use the Poincar\'e inequality
\begin{equation}\label{Poincare}
\big\|w\big\|_{\bL^2(V)}^2\leq \frac{ \big(\int_V w\, \chi\, x\, dx\,
  d\te\big)^2}{\int_V \chi^2\, x\, dx\, d\te}+
C_\de h^{2-\de}\bigg( \Big\|\chi\,\pd_x\Big(\frac w\chi\Big)\Big\|_{\bL^2(V)}^2+ \big\|D_\te w\big\|_{\bL^2(V)}^2\bigg)\,,
\end{equation}
valid for any $w\in\bH^1$ with $\de$ being an arbitrarily small positive constant and 
\[
\chi(x):=\begin{cases}
1 &\text{if } \al\geq1\,,\\
x^{-\al} &\text{if }  \al<1\,.
\end{cases}
\]
Let us now apply Eq.~\eqref{Poincare} to $w:=u_j-u_k$. The fact that
the sequence $u_j$ is weakly convergent ensures that 
\[
\ep_{jk}:=\int_V (u_j-u_k)\, \chi\, x\, dx\, d\te\to0
\]
as $j,k\to\infty$, so from~\eqref{cV} and~\eqref{Poincare} we get
\begin{align}\label{Cauchy}
\big\| u_j-u_k\big\|_{\bL^2}^2 &\leq \ep_{jk}+C_\de
h^{2-\de}\sum_{V\in \cV_h}\bigg( \Big\|\chi\,\pd_x\Big(\frac
{u_j-u_k}\chi\Big)\Big\|_{\bL^2(V)}^2+ \big\|D_\te u_j-D_\te u_k\big\|_{\bL^2(V)}^2\bigg)\,.
\end{align}
When $\al<1$, 
\[
\chi\,\pd_x\Big(\frac {u_j-u_k}\chi\Big) = \bD_x (u_j- u_k)\,,
\]
while for $\al\geq1$ we have
\begin{align*}
\sum_{V\in \cV_h}\Big\|\chi\,\pd_x\Big(\frac
{u_j-u_k}\chi\Big)\Big\|_{\bL^2(V)}^2&\leq
C\big\|\pd_x(u_j-u_k)\big\|_{\bL^2}^2\\
&\leq C\Bigg( \|\bD_x(u_j-u_k)\|_{\bL^2}+\al\bigg\|\frac
{u_j-u_k}x\bigg\|_{\bL^2} \Bigg)\\
&= C\Bigg(\|\bD_x(u_j-u_k)\|_{\bL^2}+\al\bigg\|\frac
{A(\bD_x(u_j-u_k))}x\bigg\|_{\bL^2}\Bigg)\\
&\leq C \|\bD_x(u_j-u_k)\|_{\bL^2}
\end{align*}
by Theorem~\ref{T.estA} and Proposition~\ref{P.vDv}. As one can take
$h$ arbitrarily small and $\|u_j\|_{\bH^1}\leq C$ for all $j$ by hypothesis, it
then follows from Eq.~\eqref{Cauchy} and the previous discussion that $u_j$
is a Cauchy sequence in $\bL^2$, as claimed.

Hence it only remains to prove the Poincar\'e
inequality~\eqref{Poincare}. By Corollary~\ref{C.density} it is enough
to prove it for $w=\chi\, F$, where $F\in
C^\infty([0,a]\times\SS^{n-2})$. As $F$ is smooth, for $(x,\te)$ and $(x',\te')$ in the set $V$ one can write
\begin{multline*}
F(x', \te')-F(x,\te)=\int_x^{ x'}\pd_s
F(s,\te_1,\dots,\te_{n-2})\, ds + \int_{\te_1}^{\te_1'}
\pd_sF(x',s,\te_1,\dots,\te_{n-2})\, ds\\
+\cdots +\int_{\te_{n-2}}^{\te_{n-2}'}
\pd_sF(x',s,\te_1',\dots,\te_{n-3}',s)\, ds\,.
\end{multline*}
Squaring both sides of the equation and using an elementary identity
we find
\begin{multline}\label{FF}
\big[F(x', \te')-F(x,\te)\big]^2\leq (n-1)\Bigg[\bigg(\int_x^{x'}\pd_s
F(s,\te_1,\dots,\te_{n-2})\, ds\bigg)^2 \\
+ \bigg(\int_{\te_1}^{\te_1'}\pd_s
F(x',s,\te_2,\dots,\te_{n-2})\, ds\bigg)^2 +\cdots + \bigg(\int_{\te_{n-2}}^{\te_{n-2}'}\pd_s
F(x',\te_1',\dots,\te_{n-3}',s)\, ds\bigg)^2\Bigg]\,.
\end{multline}

Suppose $\al\geq1$, so that $F=w$ and the inequality~\eqref{Poincare} is just
the ordinary Poincar\'e inequality with some control on
the dependence of the constants on the size of the domain. Let us integrate the
inequality~\eqref{FF} over $V\times V$ with respect to the measure
\begin{equation}\label{measure}
x\, dx\, x'\, dx'\, d\te\, d\te'\,.
\end{equation}
The integral of the LHS is then 
\[
\int_{V\times V} \big[F(x', \te')-F(x,\te)\big]^2\, x\, dx\, x'\, dx'\, d\te\, d\te'= 2|V|\big\|w\|_{\bL^2(V)}^2-2\bigg(\int_V w\, x\, dx\,
  d\te\bigg)^2\,,
\]
with 
$$
|V|:=\int_V x\, dx\, d\te
$$
being the measure of
the set $V$. The integral of the first summand in the RHS
of~\eqref{FF} gives
\begin{align*}
\int_{V\times V} \bigg(\int_x^{x'}&\pd_s
w(s,\te_1,\dots,\te_{n-2})\, ds\bigg)^2 x\, dx\, x'\, dx'\, d\te\,
d\te' \\
&\leq\int \bigg|\int_x^{x'}\frac{ds'}{s'}\bigg|\bigg(\int_{x_0}^{x_0+h}s\,\big[\pd_s
w(s,\te_1,\dots,\te_{n-2})\big]^2\, ds\bigg)^2 x\, dx\, x'\, dx'\, d\te\,
d\te'\\[1mm]
&=I_1\,\big\|\pd_x w\big\|_{\bL^2}\leq CI_1\,\big\|\bD_x w\big\|_{\bL^2(V)}\,,
\end{align*}
where
\begin{align*}
I_1&:=\int_V\bigg(\int_{x_0}^{x_0+h}\bigg|\int_x^{x'}\frac{ds'}{s'}\bigg|\,
x\, dx \bigg)\,x'\, dx'\,d\te'\\
&=\int_{(-h,h)^{n-2}}\bigg(\int_{x_0}^{x_0+h}\int_{x_0}^{x_0+h}x\, x'\,\big|\log x-\log
x'\big|\,  dx'\,  dx\bigg)\,d\te'\,.
\end{align*}

We shall next prove that 
\begin{equation}\label{I1}
I_1\leq C h^{2-\de}|V|\,.
\end{equation}
for any $\de>0$. For this, let us first assume that $x_0<10 h$. By the
symmetry of the integrand under the exchange of $x$ and $x'$ and the fact that $|\log x|\leq c_\de x^{-\de}$ for $x\in(0,a)$, we have then 
\[
I_1\leq 2\bigg(\int_{0}^{11h}x\, |\log x|\,
dx\bigg)\bigg(\int_V x'\, dx'\, d\te'\bigg)<C_\de h^{2-\de}|V|\,.
\]
When $x_0\geq10h$, we can use that there is
some $\bar x$ between $x$ and $x'$ such that
\begin{align*}
I_1&\leq \int_{(-h,h)^{n-2}}\bigg(\int_{x_0}^{x_0+h}\int_{x_0}^{x_0+h}\frac{|x-x'|}{\bar x}\,
x\, x'\, dx'\, dx\bigg)\, d\te'\\
&\leq h \bigg(\int_{x_0}^{x_0+h}
\frac{x}{\bar x}\, dx\bigg)\bigg(\int_Vx'\, dx'\, d\te'\bigg)\leq \frac{11}{9}h^2|V|\,,
\end{align*}
so the inequality~\eqref{I1} is proved. 

The integral of any other summand in the inequality~\eqref{FF} with
respect to the measure~\eqref{measure} can be estimated using a
similar but simpler argument as
\begin{align*}
\int_{V\times V} \bigg(\int_{\te_m}^{\te_m'}&\pd_s
w(x',\dots,\te_{m-1}',s,\dots,\te_{n-2})\, ds\bigg)^2 x\, dx\, x'\, dx'\, d\te\,
d\te' \\
&\leq\int_{V\times V} xx'|\te_m-\te_m'|\bigg(\int_{-h}^{h}\big[\pd_s
w(x',\dots,\te_{m-1}',s,\dots,\te_{n-2})\big]^2\bigg)^2 \\[1mm]
&\leq Ch^2|V|\|\pd_{\te_j}w\|_{\bL^2(V)}^2\,,
\end{align*}
which finishes the proof of the Poincar\'e inequality~\eqref{Poincare}
for $\al\geq1$. The case of $\al<1$ is analogous, the only difference
being that one must integrate the inequality~\eqref{FF} with respect
to the measure
\[
\chi(x)^2\, x\, dx\, \chi(x')^2\,x'\, dx'\, d\te\, d\te'
\]
instead of~\eqref{measure}.
\end{proof}

\section{Elliptic estimates at infinity}\label{Ell}
\label{S.elliptic}

One of the key ingredients in our proof of the well-posedness property
for the mixed initial-boundary value problem corresponding to
holography in asymptotically AdS manifolds consists in estimates in
twisted Sobolev spaces for the elliptic part $L_g$ of the hyperbolic
operator $P_g$. These are an analog for asymptotically AdS
spaces of the estimates for asymptotically Euclidean spaces derived by
Christodoulou and Choquet-Bruhat in~\cite{CC}. The statements and
method of proof, however, are totally different: in the asymptotically
Euclidean setting, the proof hinges on an integral inequality due to
Nirenberg and Walker~\cite{Nirenberg} and it is enough to consider
standard Sobolev spaces with dimension-dependent polynomial
weights. For more general elliptic operators, related estimates in weighted
Sobolev spaces can be found in~\cite{Rabier}. It should be noticed
that, in view of our future applications to energy inequalities for
the wave equations, here we need a different approach yielding
estimates in {\em twisted}\/ Sobolev spaces $\bH^m$. 

To define what we mean by the elliptic part of $P_g$, observe that, by 
Eq.~\eqref{eq-u}, the operator $P_g$ reads as
\[
P_gw=-\pd_{t}^2w+b_i\, \pd_t\pd_{\te^i}w+ b_0\, \pd_t\pd_xw + b\,
\pd_tw+ L_gw\,.
\]
The operator $L_g$, which we define as the elliptic part of $P_g$, is an elliptic operator of second order in the space
variables $(x,\te)$ (whose coefficients also depend on time) and which is of the form
\begin{align}
L_g w=-\bD_x^*\bD_x w+ \De_\te w + \cO(x)\, D_{x\te}w +\cO(x^2)\,
  D_{x\te}^2w.\label{Lg}
\end{align}
We will consider the equation 
\begin{equation}\label{LguF}
L_g w=F \quad \text{in }
(0,a)\times\SS^{n-2}
\end{equation}
with suitable Dirichlet boundary conditions
\begin{equation}\label{Dirichlet}
x^{\al}
w|_{x=0}= w|_{x=a}=0\,.
\end{equation}
When considering
weak solutions to the equations, these conditions will simply
translate as the requirement that the solution $w$ lies in $\bH^1_0$.

There are two important aspects we need to pay attention to in the derivation of the estimates at infinity for the equation~\eqref{LguF}: the way the twisted
derivative and its adjoint enter the estimates, and the dependence of
the various constants on the small parameter $a$. We will first consider the simpler problem
\begin{equation}\label{ellip1}
\bD_x^*\bD_x w-\De_\te w=F(x,\te) \quad\text{in } (0,a)\times\SS^{n-2}\,,
\end{equation}
again with the boundary conditions~\eqref{Dirichlet}, with the
obvious notion of weak solutions $w\in\bH^1_0$. To spell out the
details in the simplest case, we will say that a function $w\in
\bH^1_0$ is a {\em weak solution}\/ of the problem~\eqref{ellip1} if
\begin{equation}\label{ellip2}
\int \big(\bD_xw\, \bD_xv +\nabla_\te w\cdot \nabla_\te
v\big)\, x\, dx\, d\te= \int F v \, x\, dx\, d\te
\end{equation}
for all $v\in \bH^1_0$. Here $\nabla_\te$ corresponds to the covariant
derivative on the $(n-2)$-sphere, the dot product is also taken with
respect to the sphere metric and hereafter all integrals are taken
over $(0,a)\times\SS^{n-2}$ unless otherwise stated.

Some arguments will be made slightly easier by keeping track of the
dependence of some constants on the width~$a$ of
this chart. For this, it is convenient to replace the variable
$x\in(0,a)$ by a rescaled variable
\[
\xi:=x/a\,,
\]
which takes values in $(0,1)$. With some abuse of notation, let us
denote the twisted derivatives with respect to $\xi$ by
\[
\bD_\xi w:=\pd_\xi w+\frac\al \xi w\,, \qquad \bD_\xi^* w=-\pd_\xi
w+\frac{\al-1} \xi w
\]
and define the ordered $m^{\text{th}}$ twisted derivative
$\bD_\xi^{(m)}$ similarly (cf.\ Eq.~\eqref{D(m)}).
In terms of this new variable, Eq.~\eqref{ellip1} reads as
\begin{equation}\label{ellip1b}
\frac{\bD_\xi^*\bD_\xi w}{a^2}-\De_\te w=F \quad\text{in } (0,1)_{\xi}\times
\SS^{n-2}_\te\,,\qquad \xi^{\al}
w|_{\xi=0}=w|_{\xi=1} =0\,,
\end{equation}
so that
\begin{equation}\label{ellip2b}
\int \bigg(\frac{\bD_\xi w\, \bD_\xi v}{a^2} +\nabla_\te w\cdot \nabla_\te
v\bigg)\, \xi\, d\xi\, d\te=\int  F v \, \xi\, d\xi\, d\te
\end{equation}
for any  $v\in \bH^1_0$.

The dependence of the various estimates on the small constant~$a$ will be
described then in terms of the norms
\begin{subequations}\label{norms-a}
\begin{align}
\|w\|_{\bH^0_a}&:=\|w\|_{\bL^2}\,,\\
\|w\|_{\bH^1_a}&:=\frac{\|w\|_{\bL^2}}a+ \frac{\|\bD_\xi w\|_{\bL^2}}a
+\|D_\te w\|_{\bL^2}\,,\\
\|w\|_{\bH^m_a}&:= \frac{\|w\|_{\bL^2}}a +\sum_{l=1}^m\bigg( \|D_\te^l
w\|_{\bL^2}+ \frac{\|D_\te^{l-1}\bD_\xi w\|_{\bL^2}}a\bigg) + \sum_{l=0}^{m-2}
\frac{\|D_\te^l \bD_\xi^{(m-l)}w\|_{\bL^2}}{a^2}\,,
\end{align}
\end{subequations}
where $m\geq2$. We will also need the norms $\|\cdot\|_{\bH^m_1}$,
which are defined as above after substituting the parameter~$a$ by $1$.

The basic estimate for the simplified problem~\eqref{ellip1}
is the following, which generalizes the $\bH^2$ estimate obtained by Warnick~\cite{Warnick} in the case $\alpha<1$ to $\bH^{m+2}$ estimates in the full range $\alpha>0$:

\begin{theorem}\label{T.ellip}
Suppose $w\in \bH^1_0$ is a weak solution  of the
problem~\eqref{ellip1}, with  $F\in\bH^m$ for some
  $m\geq0$. Then $w\in \bH^{m+2}$ and satisfies the estimate
\begin{equation*}
\|w\|_{\bH^{m+2}_a}\leq c_m\|F\|_{\bH^m_1}\,,
\end{equation*}
where the constant $c_m$ does not depend on~$a$.
\end{theorem}
\begin{proof}
The $\bH^1_a$~estimate
\begin{equation}\label{H1ellip}
\frac{\|\bD_\xi w\|_{\bL^2}}a+\|D_\te w\|_{\bL^2}\leq c\|F\|_{\bL^2}
\end{equation}
follows immediately from the identity~\eqref{ellip2b} after taking $v=w$.
This implies that $\|w\|_{\bL^2}\leq ca\|F\|_{\bL^2}$ by Corollary~\ref{C.density}.

Let us now prove the $\bH^2_a$~estimates
\[
\frac{\|\bD_\xi ^* \bD_\xi  w\|_{\bL^2}}{a^2} + \frac{\|D_\te\bD_\xi  w\|_{\bL^2}}a +\|D_\te^2 w\|_{\bL^2}\leq c\|F\|_{\bL^2}\,.
\]
To explain the gist of the method, suppose we formally take $v=-Y_j^*Y_jw$ in
the identity~\eqref{ellip2b}, where $Y_j:=\pd_{\te^j}$ and $Y_j^*$ is its
formal adjoint, and sum over $j$. After integrating by parts we find that
\begin{align*}
I&:=\sum_{j=1}^{n-2}\int \bigg[\frac{\bD_\xi w\, \bD_\xi (Y_j^*Y_jw)}{a^2} +\nabla_\te w\cdot \nabla_\te
(Y_j^*Y_jw)\bigg]\, \xi\, d\xi\, d\te\\
&=\sum_{j=1}^{n-2}\int\bigg[\frac{(\bD_\xi Y_j
w)^2}{a^2}+|\nabla_\te(Y_jw)|^2+\Ga_j(D_\te w,D_\te Y_j w)\\
&\qquad \qquad \qquad \qquad \qquad \qquad \qquad \qquad \qquad +\tilde \Ga_j\bigg(\frac{\bD_\xi w}a,\frac{D_\te\bD_\xi w}a\bigg)\bigg]\, \xi\, d\xi\, d\te\,,
\end{align*}
where $\Ga_j,\tilde\Ga_j$ are bilinear functions of their entries that
depend smoothly on~$x$ and~$\te$. Therefore we find
\begin{align}
I&\geq \frac{\|D_\te\bD_\xi  w\|_{\bL^2}^2}{a^2}+ \|D_\te^2 w\|_{\bL^2}^2- c \|D_\te 
w\|_{\bL^2} \|D_\te^2 w\|_{\bL^2}- c \frac{\|\bD_\xi w\|_{\bL^2}}a \frac{\|D_\te\bD_\xi w\|_{\bL^2}}a \notag\\
&\geq c\bigg(\frac{\|D_\te\bD_\xi  w\|_{\bL^2}^2}{a^2} + \|D_\te^2 w\|_{\bL^2}^2- c \|F\|_{\bL^2} ^2\bigg)\label{I1ellip}
\end{align}
where we have used the estimate~\eqref{H1ellip} for the first
derivatives of $w$ and the elementary identity $AB\leq \de
A^2+B^2/(4\de)$. Since, by~\eqref{ellip2b},
\[
I=\sum_{j=1}^{n-2}\int F \, Y_j^*Y_jw \, \xi\, d\xi\, d\te\leq c\|F\|_{\bL^2}\|D_\te^2w\|_{\bL^2}\,,
\]
using~\eqref{I1ellip} we arrive at
\begin{equation}\label{H2ellip1}
\frac{\|D_\te\bD_\xi  w\|_{\bL^2}}a +\|D_\te^2 w\|_{\bL^2}\leq c\|F\|_{\bL^2}\,.
\end{equation}
Of course, this calculation does not make sense in this form. First of
all, one cannot take $v=Y_j^*Y_jw$ because $w$ is not, a priori,
differentiable enough. However, it is standard that this difficulty
can be overcome by replacing the partial derivatives~$\pd_{\te^j}$ in
the expression of $Y_j$ by finite differences (see e.g.~\cite{Evans}). The second problem is
that the derivatives $\pd_{\te^j}$ are not well-defined globally, as
the coordinate $\te^j$ is just local. It is well known too that this
can be circumvented either by considering a family of vector fields of the
form 
\[
Y_j=\sum_{k=1}^{n-2} Y_{jk}(\te)\, \pd_{\te^k}\,,
\] 
where $Y_{jk}$ is a smooth function supported in a coordinate patch of
the sphere $\SS^{n-2}$, and taking as many vector fields of this
form as necessary to ensure that they span the whole tangent plane at
each point of the sphere, or by using a partition of unity.  For brevity, we will skip these cumbersome
but standard details.

Once we have established the estimates~\eqref{H2ellip1}, the estimate
for $\bD_\xi ^*\bD_\xi  w$ is easy: it suffices to isolate this term in the
equation to write
\[
\frac{\|\bD_\xi ^*\bD_\xi  w\|_{\bL^2}}{a^2}=\|F+\De_\te w\|_{\bL^2}\leq \|F\|_{\bL^2}+
c\|D_\te^2 w\|_{\bL^2}\leq c \|F\|_{\bL^2}\,.
\]

The proof of the higher order estimates uses the same set of
ideas. Basically, additional regularity in $\te$ is recovered by integrating by
parts in the identity~\eqref{ellip2}. For example, if $F\in \bH^1$ one
would obtain the estimate
\[
\frac{\|D_\te^2\bD_\xi w\|_{\bL^2}}{a}+\|D_\te^3w\|_{\bL^2}\leq
c\big(\|F\|_{\bL^2}+\|D_\te F\|_{\bL^2}\big)
\]
from the identity~\eqref{ellip2} after taking $v=Y_j^*Y_k^*Y_kY_jw$, with
the same caveats as above. On the other hand,  additional (twisted)
derivatives with respect to~$x$ are recovered by isolating 
\[
\frac{\bD_\xi ^*\bD_\xi w}{a^2}=F+\De_\te w
\]
and then taking as many derivatives as needed in this equation:
\begin{align*}
  \frac{\|\bD_\xi^{(3)} w\|_{\bL^2}}{a^3}&=\|\bD_\xi  F+\De_\te \bD_\xi 
  w\|_{\bL^2}\leq \|\bD_\xi   F\|_{\bL^2}+\|\De_\te \bD_\xi   w\|_{\bL^2}\\
  &\leq c\big(\|F\|_{\bL^2}+\|D_\te F\|_{\bL^2} +\|\bD_\xi 
  F\|_{\bL^2}\big)\,.
\end{align*}
For $F$ with a larger number $m$ of derivatives in $\bL^2$, one would
repeat this process $m$~times, increasing by two the number of angular
derivatives taken in the function~$v$ and differentiating the equation
with respect to~$\xi$ after that using $\bD_\xi $ or $\bD_\xi ^*$
alternatively. The only things one has to pay attention to is that the
twisted derivative $\bD_\xi $ and its adjoint $\bD_\xi ^*$ appear in
the right places and that the powers of~$a$ do appear as in the
$\|\cdot\|_{\bH^{m+2}_a}$ norm. Details are largely straightforward and will be
omitted.
\end{proof}

The energy estimate at infinity for the full elliptic operator $L_g$
arises now as an easy corollary to Theorem~\ref{T.ellip}. We can
safely assume that the parameter $a$ is small.

\begin{corollary}\label{C.ellip}
  Let $w\in \bH^1_0$ be a weak solution of the equation $L_g w=F$ in $(0,a)\times\SS^{n-2}$,
  with $F\in \bH^m$ for some $m\geq0$. Then $w\in \bH^{m+2}$ and satisfies the estimate
\begin{equation*}
\|w\|_{\bH^{m+2}_a}\leq c_m\|F\|_{\bH^m_1}
\end{equation*}
with a constant $c_m$ independent of~$a$.
\end{corollary}
\begin{proof}
It follows from Eq.~\eqref{Lg} that
\begin{multline*}
L_gw=-\frac{\bD_\xi^*\bD_\xi w}{a^2}+ \De_\te w + \cO(1)\, \bD_\xi w +
\cO(1)\, \bD_\xi^* \bD_\xi w\\
+ \cO(a)\, D_\te \bD_\xi w+\cO(a)\,
D_\te w+ \cO(a^2)\,
  D_{x\te}^2w+ \cO(1) w\,.
\end{multline*}
From this expression and the definition of the $\bH^m_a$ norms, it then follows that the function $w$ satisfies the equation
\[
-\frac{\bD_\xi^*\bD_\xi w}{a^2}+ \De_\te w=\tilde F
\]
with a function satisfying
\begin{align*}
\|\tilde F\|_{\bH^m_1}\leq \|F\|_{\bH^{m}_1} + ca\|w\|_{\bH^{m+2}_a}\,.
\end{align*}
Hence Theorem~\ref{T.ellip} yields the estimate
\[
\|w\|_{\bH^{m+2}_a}\leq c\|F\|_{\bH^{m}_1} + ca\|w\|_{\bH^{m+2}_a}\,,
\]
which proves the result provided $a$ is small enough (that is, $a$
smaller than some positive constant independent of $F$).
\end{proof}




\section{Wave propagation at infinity}\label{WaveProp}
\label{S.local}

Now that we have proved the estimates at infinity for the elliptic part $L_g$ of the wave operator $P_g$ in an asymptotically AdS chart, we are ready to move on to the next step in our proof of well-posedness, which is concerned with solving the following Cauchy problem in the
region $\{|t|<T\}$ of an asymptotically AdS patch:
\begin{subequations}\label{CauchyAdS}
\begin{align}
&P_gv=F(t,x,\te)\quad \text{in
}(-T,T)\times (0,a)\times \SS^{n-2}\,,\\
& v(0)=v_0\,,\quad \pd_t v(0)=v_1\,.
\end{align}
\end{subequations}
Here $T$ is an arbitrary real constant and we will supplement this
problem with the homogeneous Dirichlet boundary conditions~\eqref{Dirichlet}. We will use the notation 
\[
H^k_t\bH^l:=H^k((-T,T),\bH^l)
\]
and similarly for other mixed Sobolev or Lebesgue spaces. We will say a function
$F(t,x,\te)$ belongs to the spacetime Sobolev space $\bH^m_{tx\te}$
of order $m$ if $F\in H^k_t\bH^{m-k}$ for all $0\leq k\leq m$. The
corresponding norm will be denoted by $\|\cdot\|_{\bH^m_{tx\te}}$.

Obviously, the non-standard part of Eq.~\eqref{CauchyAdS} is
hidden in the twisted derivatives $\bD_x$. The key point is that, with
some work, the energy associated with the twisted Sobolev spaces
$\bH^m$ allows us to overcome this difficulty. In what follows we will
see how one can implement this idea. Again, we assume throughout this
section that $a$ is suitably small.

\begin{theorem}\label{P.H1AdS}
For any $F\in L^1_t\bL^2$, there is a unique weak solution $v\in
L^2_t\bH^1_0\cap H^1_t\bL^2$ to the
problem~\eqref{CauchyAdS}. Moreover, it satisfies the energy estimate
\begin{equation}\label{H1Pg}
\|v\|_{L^\infty_t \bH^1_a}+\|\pd_tv\|_{L^\infty_t\bL^2}\leq c\,\e^{caT}\bigg(
\|v_0\|_{ \bH^1_a}+\|v_1\|_{\bL^2}+ \int_{-T}^T\|F(t)\|_{\bL^2}\, dt\bigg)
\end{equation}
with a constant $c$ independent of $a$ and $T$.
\end{theorem}

\begin{proof}
It follows from Eq.~\eqref{eq-u} that the equation $P_gv=F$ can be
written as
\begin{multline}\label{energye}
\pd_{t}^2v+\frac{\bD_\xi^*(b\bD_\xi
  v)}{a^2}+\pd_{\te^i}^*(\ga^{ij}\pd_{\te^j} v) +
a\pd_{\te^i}^*(b^i \pd_\xi v)+a\pd_\xi^*(b^i \pd_{\te^i} v) \\
=F+a\pd_\xi(b^0\pd_t v)+ a\pd_{\te^i}(\tilde b^i\,\pd_tv) + \cO(1)\bD_\xi v +\cO(a)\,
D_{t\te}v+\cO(1) v\,,
\end{multline}
where $\pd_{\te^i}^*$ and $\pd_\xi^*$ are the formal adjoints of the partial derivatives
$\pd_{\te^i},\pd_\xi$ with respect to the measure $\xi\,d\xi\, d\te$ and
\begin{gather*}
b=1+\cO(a^2)\,,\quad \ga^{ij}=(g_{\SS^{n-2}})^{ij}+\cO(a^2)\,,\quad
b^0,b^i,\tilde b^j=\cO(1)\,.
\end{gather*}
Notice that $b^0$, $b^i$ and $\tilde b^j$ vanish on $\{\xi=0\}$.
Let us define the energy as
\begin{equation}\label{energy}
E[v](t):=\int_{(0,1)\times\SS^{n-2}} \bigg(v_t^2+\frac{b}{a^2}(\bD_\xi
  v)^2+\ga^{ij}\pd_{\te^i} v\, \pd_{\te^j}v+ 2a b^i \, \pd_\xi v\, \pd_{\te^i}
  v\bigg)\, \xi\, d\xi\, d\te\,.
\end{equation}
Since $a$ is small, a simple computation shows that $E(t)$ is
equivalent to the (squared) $\bH^1_a$ norm of $v(t)$ in the sense that
\begin{equation}\label{vE}
c_1  E[v](t)^{\frac12}\leq \|v(t)\|_{\bH^1_a}\leq c_2 E[v](t)^{\frac12}
\end{equation}
with constants independent of $a$.

To derive the energy estimates, in this section we will take derivatives of $v$ as if
it were smooth. It is standard that this can be justified either by
considering a smooth $v$ and the density property proved in Corollary~\ref{C.density} or
an argument using Garlerkin's method, in which we expand $v$ (e.g.) in the basis of
eigenfunctions of the operator $a^{-2}\bD_\xi^*\bD_\xi-\De_\te$ in
$(0,a)\times\SS^{n-2}$ with
Dirichlet boundary conditions. 

Using the boundary conditions to integrate by parts and the equation
satisfied by~$v$, a straightforward computation shows that
\begin{multline*}
\pd_t E[v](t)=2\int v_t\Big(F -a\pd_\xi(b^0v_t)-
a\pd_{\te^i}(\tilde b^iv_t)\\
 +\cO(1)\bD_\xi v +\cO(a)\,
D_{t\te}v+\cO(1) v\Big)\, \xi\, d\xi\, d\te\,.
\end{multline*}
All the terms but the second and third ones can be immediately
controlled as
\begin{multline*}
\int v_t\Big(F +\cO(1)\bD_\xi v +\cO(a)\,
D_{t\te}v+\cO(1) v\Big)\, \xi\, d\xi\, d\te\\
\leq
\|F(t)\|_{\bL^2}E[v](t)^{\frac12}+ caE[v](t)\,,
\end{multline*}
where we have used that $\|v(t)\|_{\bL^2}\leq C\|\bD_\xi
v(t)\|_{\bL^2}\leq caE(t)^{1/2}$ by Proposition~\ref{C.density}. The
second term can be estimated as
\begin{align*}
\int v_t\pd_\xi(b^0v_t) \, \xi\, d\xi\, d\te&=\int \pd_\xi
b^0\,v_t^2\, \xi \,d\xi\, d\te+\frac 12\int b^0 \pd_\xi(v_t^2) \, \xi\,
d\xi\, d\te\\
&=\int \Big[\xi\,\pd_\xi
b^0-\frac12\pd_\xi(\xi b^0)\Big]v_t^2\, d\xi\, d\te\leq c E[v](t)\,,
\end{align*}
where we have used that $b^0$ vanishes at $\xi=0$. The third term can
be controlled in a similar fashion, so we arrive at 
\begin{equation}\label{varEv}
\pd_t E[v](t)\leq c\|F(t)\|_{\bL^2}E[v](t)^{\frac12}+ caE[v](t)\,,
\end{equation}
which is well known to imply that
\[
E[v](t)^{\frac12}\leq c\bigg(E[v](0)^{\frac12} +\int_0^t\|F(\tau)\|_{\bL^2}d\tau\bigg)\e^{cat}
\]
by Gronwall's inequality. (Here and in what follows, we follow the
convention that the above integral is nonnegative independently of the
sign of $t$.)

Due to the equivalence of norms~\eqref{vE}, this yields the desired energy estimate~\eqref{H1Pg}. 
It is standard that this energy estimate readily implies
that there is at most a unique weak solution in $L^2_t\bH^1\cap
H^1_t\bL^2$ to the equation (see e.g.~\cite{Lady}). It also ensures that Garlerkin's method~\cite{Lady} converges
to a weak solution of the problem~\eqref{CauchyAdS}, thereby proving
the existence of a unique solution to the problem.
\end{proof}


As is well known from the case of bounded domains in Euclidean space,
to obtain higher order estimates we need to impose some compatibility
conditions on the source function $F$ and the initial conditions, in
addition to the regularity assumptions. In the case of
$\bH^{m+1}$ estimates, the compatibility conditions correspond to the
vanishing of the trace to the boundary of the functions
\begin{equation}\label{vj}
v_j:=\pd_t^jv(0)
\end{equation}
for all
$j\leq m$ (that is, $\cB_jv\in\bH^1_0$). Using Eq.~\eqref{CauchyAdS}, it is
apparent that $v_j$ can be written in terms of the initial
conditions $v_0,v_1$ and the function $F$. A more careful look at the
algebraic structure of these functions reveals that 
\begin{equation}\label{vj2}
\|v_j\|_{\bH^k}\leq C_{j,k}\bigg(\sum_{l=0}^{j-2}
\|F\|_{W^{l,\infty}_t\bH^{k+j-2-l}}+ \|v_0\|_{\bH^{k+j}}+\|v_1\|_{\bH^{k+j-1}}\bigg)\,.
\end{equation}



In the next theorem, we obtain mixed energy estimates on the higher-order time derivatives of $v$ solving Eq.~\eqref{CauchyAdS}, assuming that the source term $F$, the boundary data $v_0,\,v_{1}$ and their higher time derivatives belong to suitably chosen Sobolev classes. 
\begin{theorem}\label{T.higherwave}
Let us fix an integer $m\geq1$. Assume that $F\in \bH^m_{tx\te}$,
$v_0\in \bH^{m+1}$, $v_1\in \bH^m$ and
that $v_j\in \bH^1_0$ for all $0\leq j\leq m$. Then the solution to
the problem~\eqref{CauchyAdS} is of class
\[
v\in \bH^{m+1}_{tx\te}\cap H^m_t \bH^1_0
\]
and satisfies
\begin{multline*}
\sum_{l=0}^{m+1}\big\|\pd_t^lv\big\|_{L^\infty_t \bH^{m+1-l}}\leq C_T\bigg(\|\pd_t^mF\|_{L^2_t\bL^2} +\sum_{j=0}^{m-1}\|\pd_t^jF\|_{L^\infty_t \bH^{m-1-j}} \\
+\|v_0\|_{\bH^{m+1}}+\|v_1\|_{\bH^m}\bigg)\,.
\end{multline*}
\end{theorem}

\begin{proof}
In fact, we will prove the more precise estimate
\begin{multline}\label{est-prop53}
\sum_{l=0}^{m+1}\big\|\pd_t^lv\big\|_{L^\infty_t \bH^{m+1-l}_a}\leq
c_T\,\bigg(\|\pd_t^mF\|_{L^1_t\bL^2} +\sum_{j=0}^{m-1}\|\pd_t^jF\|_{L^\infty_t \bH^{m-1-j}_1} \\
+\sum_{j=0}^{m-1}\|v_j\|_{\bH^1_a}+\|v_{m+1}\|_{\bL^2}\bigg)\,,
\end{multline}
where the constant is independent of $a$ and depends nicely on $T$.
Let us first prove the result for $m=1$, using the same notation as in
the proof of Proposition~\ref{P.H1AdS} without further notice. 

Consider the energy
\[
E_1(t):=E[v](t)+E[v_t](t)\,.
\]
The equation~\eqref{CauchyAdS} implies that for a.e.~$t$ the function $v$ satisfies the
elliptic equation
\[
L_gv=-v_{tt}+F+a\pd_\xi(b^0\pd_t v)+ a\pd_{\te^i}(\tilde b^i\,\pd_tv) + \cO(1)\bD_\xi v +\cO(a)\,
D_{t\te}v+\cO(1) v\,,
\]
whose RHS can be estimated in $\bL^2$ norm by
\[
c(E_1(t)^{\frac12}+\|F(t)\|_{\bL^2})
\]
on account of the equivalence~\eqref{vE} and the bounds for the
coefficients of the equation. By Theorem~\ref{T.ellip}, this yields
the a priori estimate 
\begin{equation}\label{ellipH2wave}
\|v(t)\|_{\bH^2_a}\leq c(E_1(t)^{\frac12}+\|F(t)\|_{\bL^2})
\end{equation}

Let us now compute the variation of $E_1(t)$. One can easily check
that
\begin{multline*}
\pd_t E_1(t)=2\int v_{tt}\Big(\pd_tF-[\pd_t,P_g]v -a\pd_\xi(b^0v_{tt})-
a\pd_{\te^i}(\tilde b^iv_{tt})\\
  +\cO(1)\bD_\xi v_t +\cO(a)\,
 D_{t\te}v_t+\cO(1) v_t\Big)\, \xi\, d\xi\, dt
\end{multline*}
Since
\[
\big\|[\pd_t,P_g]v(t)\big\|_{\bL^2}\leq ca\|v(t)\|_{\bH^2_a}\leq ca(E_1(t)^{\frac12}+\|F(t)\|_{\bL^2})
\]
by Eq.~\eqref{ellipH2wave} and the other terms are basically as in
Proposition~\ref{P.H1AdS}, one can then use
Eq.~\eqref{varEv} and argue as in the proof of the aforementioned
proposition, mutatis mutandis, to find that
\begin{equation}\label{secondeq}
\pd_t E_1(t)\leq ca E_1(t)+c\big( \|F(t)\|_{\bL^2}+
\|\pd_tF(t)\|_{\bL^2}\big)\, E_1(t)^{\frac12}\,,
\end{equation}
This readily yields 
\[
E_1(t)^{\frac12}\leq c\,\e^{cat}\bigg(E_1(0)^{\frac12}+\int_0^t\big(\|F(\tau)\|_{\bL^2}+ \|\pd_tF(\tau)\|_{\bL^2}\big)\,d\tau\bigg)\,,
\]
which, when combined with the elliptic estimate~\eqref{ellipH2wave}
and the norm equivalence~\eqref{vE}, proves the
estimate~\eqref{est-prop53} for $m=1$.

The general case is proved using the same
ideas. Basically, one considers the energies
\[
E_k(t):=\sum_{j=0}^k E[\pd_t^jv](t)\,,
\]
with $k\leq m$. By taking time derivatives in the equation and
using elliptic estimates as in~\eqref{ellipH2wave}, one finds that for
a.e.~$t$ we have
\[
\sum_{j=0}^{k-1} \|\pd_t^jv(t)\|_{\bH^{k+1-j}_a}\leq
  c\bigg(E_k(t)+\sum_{j=0}^{k-1}\|\pd_t^jF(t)\|_{\bH^{k-1-j}_1}\bigg)\,.
\]
Arguing now as in the proof of~\eqref{secondeq}, one then finds that
$E_k$ satisfies the estimate
\begin{equation}\label{pdtEk}
\pd_t E_k(t)\leq caE_k(t)+ c E_k(t)^{\frac12}
\sum_{j=0}^{k}\|\pd_t^jF(t)\|_{ \bH^{k-j}_1}\,,
\end{equation}
to which we can apply Gronwall's inequality. This yields the desired estimates arguing as above. The details are omitted.
\end{proof}

Our final result for this section gives us estimates for the time evolution of the $\bH^{k}$ norm of $v$ and $\partial_{t}v$ again in terms of $F$ and the initial data $v_{0}\,v_{1}$. 
\begin{proposition}\label{P.HkEvolution}
Under the same hypotheses of Theorem~\ref{T.higherwave}, there is a
constant $c$ that does not depend on $T$ such that for a.e.~$t$
\[
\|v(t)\|_{\bH^{m+1}}+\|\pd_tv(t)\|_{\bH^m}\leq c\,
\e^{ca|t|}\bigg(\|v_0\|_{\bH^{m+1}} + \|v_1\|_{\bH^{m}} +
\int_0^t\|F(\tau)\|_{\bH^m}\, d\tau\bigg)\,.
\]
\end{proposition}

\begin{proof}
With the notation of the proof of Theorem~\ref{P.H1AdS} (cf.\ Eq.~\eqref{energye}), let us
set
\begin{multline*}
e_m(t):=\sum_{j+|\be|\leq m}\int_{(0,1)\times\SS^{n-2}} \bigg[
\big(\bD_\xi^{(j)}D_\te^\be v_t\big)^2 +\frac b{a^2}
\big(\bD_\xi^{(j+1)}D_\te^\be v\big)^2 \\
+\ga^{ij} \, 
\bD_\xi^{(j)} D_\te^\be\pd_{\te^i}v \,
\bD_\xi^{(j)} D_\te^\be\pd_{\te^j}v +2ab^i \,
\bD_\xi^{(j)}D_\te^\be\pd_\xi v\, \bD_\xi^{(j)}D_\te^\be
\pd_{\te^i}v\bigg]\, \xi\, d\xi\, d\te\,,
\end{multline*}
where $\be=(\be_1,\dots,\be_{n-2})$ denotes a multiindex and, as usual, we are abusing the
notation for angular derivatives. One should compare this expression
with that of the energy $E[v](t)$, see Eq.~\eqref{energy}. Arguing as in
the case of $E[v](t)$, it is easy to see that $e_m(t)$ satisfies
\begin{equation}\label{emHm}
\frac{e_m(t)}{c}\leq \|v(t)\|_{\bH^{m+1}}^2+\|\pd_tv(t)\|_{\bH^m}^2\leq
c\, e_m(t)\,.
\end{equation}

A tedious but straightforward
computation similar to the ones carried out in the proof of
Theorem~\ref{T.higherwave} shows that the derivative of $e_m$ can be
estimated by
\[
\pd_t e_m(t)\leq c e_m(t)+ce_m(t)^{\frac12}\,\|F(t)\|_{\bH^m}\,.
\]
Since the commutator 
\[
[\bD_\xi,\bD_\xi^*]w=\frac{1-2\al}{\xi^2}
\]
is singular at $\xi=0$, the key point to check in order to derive this
inequality is that all the twisted derivatives appear with the right
ordering. Once the differential inequality has been established,
the proposition follows from the norm equivalence~\eqref{emHm} and
Gronwall's inequality.
\end{proof}

\section{Peeling off large solutions near the conformal boundary}
\label{S.boundary}

We now move on the case of solutions of the wave equations~\eqref{LW}
and~\eqref{NLW} with nontrivial boundary conditions at the conformal
infinity. Again, we work in an asymptotically AdS chart and our
specific goal for this section is to show how we can reduce the
analysis of a solution with nontrivial boundary conditions at infinity
to that of the sum of certain terms that are large at infinity, but
essentially controlled, and another solution that depends on the
boundary datum in a complicated way but is smaller at infinity. We
refer to this method of passing from the boundary datum to this sum of
large but controlled terms and a smaller function, which is described in
Theorem~\ref{T.large}, as a peeling off of the large solutions near
the conformal boundary. 

We recall that the nonlinearities that we are considering in
Eq.~\eqref{NLW} are of the form~\eqref{nonlinF}. Throughout this paper
we will assume that the exponent $q$ of this nonlinearity satisfies
\[
q\geq \al+2\,.
\]
More general nonlinearities can be dealt with using the
same ideas, but this choice of $F(\nabla\phi)$ will allow for a more
concrete presentation. 

The nonlinear wave equation~\eqref{NLW} with
nonlinearity~\eqref{nonlinF} reads in an asymptotically AdS
chart as
\[
P_g u=Q(u,u)\,,
\]
where we have used the relationship~\eqref{BoxPg} to write the
nonlinear term as
\[
Q(u,v):=x^{q-\frac{n+3}2}\, \widetilde\Ga(t,x,\te)\,
g^{\mu\nu}\,\pd_\mu\big(x^{\frac{n-1}2}u\big)\, \pd_\mu\big(x^{\frac{n-1}2}v\big)\,.
\]
Observe that $\widetilde\Ga$ is smooth up to the boundary (that is, in
$\RR\times[0,a]\times\SS^{n-2}$) and that Greek subscripts run over the set
$\{t,x,\te^1,\dots, \te^{n-2}\}$. The linear wave equation corresponds
to the case where $Q$ is identically zero, cf.~Eq.~\eqref{LW}.

To describe the kind of terms that appear in the equations, we will
introduce the notation $\cO\pol(x^s h^{\leq k})$ to denote functions
that are of order $x^s$ and whose dependence on $(t,\te)$ can be
controlled in terms of a polynomial of the first $k$ derivatives of certain
function $h(t,\te)$. Specifically, given a smooth enough function
$h(t,\te)$, a real $s$ and an integer $k$, we will say that a function
$H(t,x,\te)$ is of order $\cO\pol(x^s h^{\leq k})$ if it can be
written as a finite sum
\[
H(t,x,\te)=\sum_{i=1}^{N_H} a_i(x)\, H_i(t,\te)
\]
in which the summands obey the bounds
\begin{equation}\label{polbounds}
\big|\pd_x^la_i(x)\big|\leq C_{H,l}x^{s-l}\,,\qquad \big|D_{\te
  t}^lH_i(t,\te)\big|\leq C_{H,l}\sum_{|J|\leq l+k}
\prod_{i=1}^{M}|D_{t\te}^{J_i} h(t,\te)|\,.
\end{equation}
Here 
\[
J=(J_1,\dots,J_{M})
\]
is a set of $M$ nonnegative integers, with $M=M_{H,l}$ also depending
on $H$ and the number of derivatives considered, and $|J|$ stands for
the sum of these integers. A crucial property of the summands is that,
by the Moser inequalities, whenever $l+k>\frac{n-1}2$ we have
\begin{equation}\label{boundHi}
\|H_i\|_{H^l_{t\te}}\leq  C_l(M)\, C_{H,l}\,\big(\|h\|_{H^{l+k}_{t\te}} + \|h\|_{H^{l+k}_{t\te}}^M\big)\,.
\end{equation}

When the bounds~\eqref{polbounds} hold with $M=1$, we drop the
subscript to write $\cO(x^sh^{\leq k})$. This case is particularly
relevant because it appears in the analysis of the linear wave
equation. Of course, in this case the estimate~\eqref{boundHi} is
valid for all $l\geq0$.

In the following lemma we describe the basic step of the peeling off
procedure. For simplicity of exposition, {\em we henceforth assume that $2\al$ is not an
  integer}\/, although it will be clear from the proofs that all the
results we prove in this section remain valid when $2\al\in\NN$ (and
for real values of $r,s$) provided one includes suitable logarithmic
terms in the statements when necessary (see Remark~\ref{R.log}
below). The point here is to observe the powers of $x$ that appear in
the different terms and how the implicit constants (namely $C_{H,l}$, $N_H$ and $M=M_{H,l}$, in
the above notation for the function $H\in \cO\pol(x^s h^{\leq k})$)
are controlled in terms of the initial functions $F$ and $G$. To be
more precise, in this section we will say that certain quantity $H\in \cO\pol(x^s h^{\leq
  k})$ is {\em controlled}\/ in terms of $F\in \cO\pol(x^{s'} h^{\leq k'})$ if
the constants of $H$ can be estimated in terms of those of $F$ as
\begin{align*}
C_{H,l}+ M_{H,l}+N_H\leq C\Big(\sup_{0\leq j\leq l}C_{F,j},\sup_{0\leq
  j\leq l}M_{F,j},N_F\Big)\,,
\end{align*}
where the function $C(\cdot,\cdot,\cdot)$ is independent of $H$ and $F$.

\begin{lemma}\label{L.reduction}
Suppose that the function $v$ satisfies an equation of the form
\begin{equation}\label{eq.Lred}
P_g v+Q(v,G)-Q(v,v)=F\,,
\end{equation}
with
\begin{align*}
F\in \cO\pol(x^{-\al+r}h)\,,\qquad 
G\in \cO\pol(x^{-\al+s}\tilde h)\,, 
\end{align*}
for some nonnegative integers $r,s$. Then one can take a function
\begin{equation}\label{eq.Lred2}
v_1:=v+\cO\pol(x^{-\al+r+2}h)+\cO\pol(x^{-\al+r+2}\tilde h)
\end{equation}
that satisfies an equation of the form
\begin{multline}\label{eq.Lred3}
P_g v_1+Q\big(v_1,\cO\pol(x^{-\al+s}\tilde h)+\cO\pol(x^{-\al+r+2}
h)\big)-Q(v_1,v_1)\\
=\cO\pol(x^{-\al+r+2}h^{\leq2})+\cO\pol(x^{-\al+r+2}\tilde
h^{\leq2})\,.
\end{multline}
Furthermore:

\begin{enumerate}
\item All the terms of the form $\cO\pol(x^s h^{\leq k})$ (resp.\
  $\cO\pol(x^s \tilde h^{\leq k})$) are controlled in terms of $F$
  (resp.\ $G$) in the sense specified above.

\item 
If $v$ satisfies the boundary condition $x^\al v|_{x=0}=0$, then we also have $x^\al v_1|_{x=0}=0$. 

\item If $F$ and $G$ are
supported in the region $\{x<a_0\}$, then one can ensure that the
various terms $\cO\pol(\cdots)$ arising in Eqs.~\eqref{eq.Lred2} and~\eqref{eq.Lred3}
are also supported in this region. 

\item If the equation is
linear (i.e., $Q:=0$), the statement remains valid with
$\cO\pol(\cdots)$ replaced by $\cO(\cdots)$.
\end{enumerate}
\end{lemma}
\begin{proof}
It is clear that
$F$ can be written as
\[
F=x^{-\al+r}H_1(t,\te)+ x^{-\al+r+2}H_2(t,\te)+\cO\pol(x^{-\al+r+2}h)\,,
\]
where the functions $H_j$ are bounded by powers of $h$ as in
Eq.~\eqref{polbounds}. Let us denote by $\chi(x)$ a smooth function
supported in $\{x<a_0\}$ and equal to $1$ in $\{x<a_0/2\}$. Setting
\[
v_0:=v+\frac{\chi(x)\,x^{-\al+r+2} H_1(t,\te)}{\al^2-(\al-r-2)^2} +
\frac{\chi(x)\, x^{-\al+r+3} H_2(t,\te)}{\al^2-(\al-r-3)^2}
\]
and making use of the elementary identity
\[
\bD_x^*\bD_x(x^s)=(\al^2-s^2)x^{s-2}\,,
\]
we find after a short computation that
\begin{multline*}
P_gv_0-Q\big(v_0, \cO\pol(x^{-\al+s}\tilde
h)+\cO\pol(x^{-\al+r+2}h)\big) + Q(v_0,v_0)\\
= \cO\pol(x^{-\al+r+2}h^{\leq 2})+\cO\pol(x^{-\al+r+s}h\tilde h) +
\cO\pol\big(x^{-\al+r+s+2}(h\tilde h)^{\leq 1}\big)\,.
\end{multline*}
Using that
\[
\cO\pol\big(x^s(GH)^{\leq k}\big)\in \cO\pol(x^sG^{\leq k}) +  \cO\pol(x^sH^{\leq k})
\]
we can now write the RHS of this equation as
\[
x^{-\al+r} H_2(t,\te)+ x^{-\al+r+1} H_3(t,\te)+
\cO\pol(x^{-\al+r+2}h^{\leq 2})+ \cO\pol\big(x^{-\al+r+s+2}\tilde h^{\leq 1}\big)\,,
\]
where $H_2$ and $H_3$ (which can be identically zero) can be controlled pointwise in terms of powers
of $h$ and $\tilde h$. If we now define
\[
v_1:=v_0+\frac{\chi(x)\,x^{-\al+r+2} H_2(t,\te)}{\al^2-(\al-r-2)^2} +\frac{\chi(x)\,x^{-\al+r+3} H_3(t,\te)}{\al^2-(\al-r-2)^2}\,,
\]
we readily find that $v_1$ satisfies an equation of the
form~\eqref{eq.Lred3}. The statement about the boundary conditions of
$v_1$ and the support of the various terms $\cO\pol(\cdots)$ stems
from the construction, as does the fact that the statement remains
valid when $Q$ is zero and all the subscripts are dropped in the terms
$\cO\pol(\cdots)$. 
\end{proof}

\begin{remark}\label{R.log}
The argument is actually valid for all real values of $\al$ and $r$,
the only difference being that when $\al=-\al+r+m$ for some
nonnegative integer $m\leq3$, the summand that we use to construct $v_1$
or $v_2$ whose denominator would apparently be singular is replaced by
a term $\log x\, x^\al H_j(t,\te)$.
\end{remark}

After describing the basic step of the peeling off procedure, we are
ready to state and prove the main result of this section:

\begin{theorem}\label{T.large}
Suppose that $u$ satisfies the equation
\[
P_gu=Q(u,u)
\]
with boundary condition $x^\al u|_{x=0}=f(t,\te)$. Then, for any positive
integer $k$, one can take a function 
\[
u_k=u-\sum_{j=0}^{k-1}\cO\pol(x^{-\al+2j} f^{\leq 2j})
\]
that satisfies an equation of the form
\[
P_g u_k=Q(u_k,u_k)+ Q\bigg(u_k,\sum_{j=0}^{k-1}\cO\pol(x^{-\al+2j} f^{\leq
  2j})\bigg)+ \cO\pol(x^{-\al+2k-2} f^{\leq2k})
\]
and the homogeneous boundary condition $x^\al
u_k|_{x=0}=0$. Furthermore,

\begin{enumerate}
\item All the terms $\cO\pol(x^{-\al+r}f^{\leq 2j})$ that appear can be
assumed to be supported in the region $\{x< a_0\}$ for any fixed
$a_0$. 
\item The implicit constants that appear in the terms
$\cO\pol(x^sf^{\leq k})$ admit bounds independent of $f$. 

\item If the equation is
linear ($Q:=0$), the statement remains valid with
$\cO\pol(\cdots)$ replaced by $\cO(\cdots)$.
\end{enumerate}
\end{theorem}

\begin{proof}
Let $\chi(x)$ be a smooth function supported in $\{x<a_0\}$ that is
identically equal to $1$ is $\{x<a_0/2\}$.
Setting
\[
u_1:=u-\chi(x)\, x^{-\al}f(t,\te)
\]
we immediately find that $x^\al u_1|_{x=0}=0$ and
\[
P_g
u_1=Q(u_1,u_1)+Q\big(u_1,\cO\pol(x^{-\al}f)\big)+\cO\pol(x^{-\al}f^{\leq
  2})\,,
\]
where the terms $\cO\pol(x^{-\al}f)$ and $\cO\pol(x^{-\al}f^{\leq 2})$
are supported in $\{x<a_0\}$. This proves the statement for $k=1$. The
general case then follows by repeatedly applying
Lemma~\ref{L.reduction} to this equation.
\end{proof}



\section{Holography for the linear wave equation}
\label{S.LW}

We are now ready to show the well-posedness of the
Klein--Gordon equation in a general asymptotically AdS manifold. Our
well-posedness theorem corresponding to the holographic prescription
in this geometric setting (Theorem~\ref{T.LW}) will follow by
combining the energy estimates obtained in Theorems~\ref{P.H1AdS}
and~\ref{T.higherwave} for the wave propagation at infinity, the
peeling properties obtained for the solutions of the linear problem
near the conformal boundary obtained in Theorem~\ref{T.large}, and
standard estimates for the wave equation away from the conformal boundary.

First of all, let us give a precise definition of what an
asymptotically AdS spacetime is. We recall that a Lorentzian manifold
$\cM$ is said to be {\em asymptotically AdS}\/ (cf.\ e.g.~\cite{HT})
if there is a spatially compact set $\cK$ such that $\cM\minus \cK$ is
the union of asymptotically AdS patches. Therefore, $\cM$ is covered
by a finite number of coordinate charts
$\cU_1,\dots,\cU_N,\cV_1,\cdots, \cV_M$, where the patches $\cU_j$ are
asymptotically AdS in the sense of Definition~\ref{D.aAdS} and the
patches $\cV_k$ are contained in~$\cK$ (and therefore ``regular''). We
also assume that $\cM$ is time-oriented, meaning that there is a
nowhere vanishing timelike vector field $\cT$, and that this vector field
coincides with the partial derivative $\pd_t$ in each asymptotically
AdS patch. For simplicity, we will also assume that the time function
$t$ is globally defined in $\cM$, satisfies $\cT t=1$ and foliates
$\cM$ as
\[
\cM=\bigcup_{\tau\in\RR}\cM_\tau\,,\qquad \cM_\tau:=\big\{t=\tau\big\}\,.
\]
We will also use the notation 
\[
\cM_{(T_1,T_2)}=\bigcup_{T_1<\tau<T_2}\cM_\tau
\]
for the part of the manifold sandwiched between two time
slices. Notice that we are not making the assumption that the defining
function $x$ of the conformal boundary is defined in the whole
manifold.

When solving the linear wave equation in $\cM$, the
finite speed of propagation ensures that, if $\phi$ solves the equation
\begin{subequations}\label{eq-cM}
\begin{gather}
\square_g\phi-\mu\phi=F\quad\text{in } \cM\,,\\
\phi|_{\cM_0}=\phi_0\,,\qquad \cT\phi|_{\cM_0}=\phi_1
\end{gather}
\end{subequations}
with $F,\phi_0,\phi_1$ supported in the coordinate chart $\cV_k$ (resp.\ $\cU_j$), for time
$t$ in some small enough interval $(-T,T)$ the function $\phi(t)$ is
supported in a lightly larger set $\cV_k'$ that is still contained in
$\cK$ (resp.\ in another slightly larger asymptotically AdS chart
$\cU_j'$). To exploit this propagation property to immediately derive
global estimates for the wave equation in $\cM$, let us define the
twisted Sobolev norm in $\cM$ at time
$t$, denoted
\[
\|\phi(\tau)\|_{\bH^k(\cM_t)}\,,
\]
as the sum of the usual $H^k$ norm of $\phi(\tau)$ in the spacelike regions
$\cV_k\cap \cM_\tau$ and the twisted norm $\|u(\tau)\|_{\bH^k(\cU_j)}$
corresponding to the asymptotically AdS charts (again, the
relationship between $\phi$ and $u$ in the asymptotically AdS chart is
given by~\eqref{phiu}). The key here is that, near the conformal
boundary, the norm we take for $\phi$ is obtained by transplanting the norm of
the associated function $u$ which we can control with our estimates in an
asymptotically AdS chart. By the propagation properties of the wave
equation, it then follows
that the estimates proved in Theorems~\ref{P.H1AdS}
and~\ref{T.higherwave}, together with the well known results for the
wave equation in globally hyperbolic spaces, yield $\bH^k(\cM_t)$-energy
estimates for the linear wave equation in $\cM$ analogous to those
proved in asymptotically AdS charts.

In view of the applications to Physics, an important technical
observation, due to Holzegel and Warnick~\cite{HW}, is that one can
also allow for the presence of black holes in $\cM$ as long as these
black holes are stationary and one can derive suitable energy
estimates for the wave equation using a redshift argument
(see~\cite{redshift,Tataru}). This corresponds to taking in the
definition of our manifold $\cM$ not only ``regular'' coordinate
patches $\cV_j$ and asymptotically AdS patches $\cU_k$, but also
stationary black hole patches. The Sobolev spaces $\bH^k(\cM_\tau)$
defined above are now defined by adding another term that is read off
from the energy inequalities in stationary black holes. Details,
including the definition of stationary black hole, will not be needed
in this paper and can be found e.g. in~\cite{HW,Tataru}. Note that the
timelike Killing vector of each black hole patch will simply be the
vector field $\cT$.

Let us summarize the conclusion of the preceding discussion in the
following theorem, where of course $\bH^1_0(\cM_\tau)$ stands for the
completion of $C^\infty\cpt(\cM_\tau)$ in the $\bH^1(\cM_\tau)$
norm. The corresponding spacetime norm $\bH^k(\cM)$ is defined in the
obvious fashion and we set
\[
\|\phi\|_{L^\infty\bH^k(\cM_{(T_1,T_2)})}:=\sup_{T_1<\tau<T_2}\|\phi\|_{\bH^k(\cM_\tau)}\,.
\]
In order to derive higher-order estimates, we need to
assume that the functions $v_j$ defined in Eq.~\eqref{vj} vanish at
the conformal boundary, which is the set $\{x=0\}$ in each
asymptotically AdS chart. An equivalent, more convenient way of doing
this more invariantly is to demand that
$\cT^j\phi|_{\cM_0}\in\bH^1_0(\cM_0)$. Of course, this quantity can be
written in terms of the initial conditions $\phi_0,\phi_1$ and the
function $F$ alone.

\begin{theorem}\label{T.wave} 
  Let us fix an integer $m\geq0
$. Assume that $F\in \bH^m(\cM)$,
  $\phi_0\in \bH^{m+1}(\cM_0)$, $\phi_1\in \bH^m(\cM_0)$ and that $\cT^j\phi|_{\cM_0}\in
  \bH^1_0(\cM_0)$ for all $0\leq j\leq m$. Then the
  problem~\eqref{eq-cM} has a unique solution in $\cM$, which is of
  class $\bH^{m+1}(\cM_{(-T,T)})$ for any $T$ and satisfies the estimates
\begin{multline}\label{est-1}
\sum_{l=0}^{m+1}\big\|\cT^l\phi\big\|_{L^\infty\bH^{m+1-l}(\cM_{(-T,T)})}\leq
C_{T}\bigg(\sum_{j=0}^{m-1}\|\cT^jF\|_{L^\infty\bH^{m-1-j}(\cM_{(-T,T)})} \\
+\|\cT^mF\|_{\bL^2(\cM_{(-T,T)})} 
+\|\phi_0\|_{\bH^{m+1}(\cM_0)}+\|\phi_1\|_{\bH^m(\cM_0)}\bigg)\,.
\end{multline}
Furthermore, for a.e.~$\tau$ we also have
\begin{multline}\label{est-2}
  \|\phi\|_{\bH^{m+1}(\cM_\tau)}+ \|\cT \phi\|_{\bH^{m}(\cM_\tau)}\leq C\,
  \e^{C|\tau|}\bigg(\|\phi_0\|_{\bH^{m+1}(\cM_0)}+ \| \phi_1\|_{\bH^{m}(\cM_0)}
  \\
  + \int_0^\tau \|F\|_{\bH^m(\cM_s)}\, ds\bigg)\,,
\end{multline}
where the constant $C$ is independent of $\tau$.
\end{theorem}

\begin{proof}
  The existence of a unique solution to the problem~\eqref{eq-cM} in
  $\cM_{(-T,T)}$ and the fact that it satisfies the
  estimate~\eqref{est-1} in this set follows from the above argument
  provided $T$ is small enough. Notice that the smallness of $T$ is
  only used to ensure that for all times $|t|<T$, the support of the solution to the
  problem \eqref{eq-cM} with $\phi_0$, $\phi_1,$ and $F$ supported in
  the set $\cV_k$ is contained in the slightly larger set $\cV_k'$,
  and similarly with data supported in $\cU_j$. This means that we can
  choose $T$ to be independent of the particular choice of the
  data. Therefore, it is standard that we can repeat the argument,
  now using the functions $(\phi|_{\cM_T},\cT\phi|_{\cM_T})$ as initial
  conditions on $\cM_T$, to eventually derive that the solution exists
  globally and satisfies the bound~\eqref{est-1} for any finite $T$.

The estimate~\eqref{est-2} is then an immediate consequence of
Proposition~\ref{P.HkEvolution} (and of the standard energy estimates for
wave equations in the non-asymptotically AdS charts $\cV_k$).
\end{proof}

Note that the RHS of the inequality~\eqref{est-1} is clearly finite
for $F\in \bH^m(\cM)$. Our next objective is to consider the case
where we have nontrivial boundary conditions at some of the ends of
the asymptotically AdS manifold, as in Eq.~\eqref{BCs2}. 

To begin with, let us make precise the meaning of these boundary
conditions. To thins end, let us recall that $\cM$ has $N$
asymptotically AdS ends, which are covered by patches
$\cU_1,\dots,\cU_N$. Given a function
\begin{equation}\label{datumf}
f=\big(f_1(t,\te),\dots, f_N(t,\te)\big)\,,
\end{equation}
we will say that the field $\phi$
satisfies the boundary condition
\begin{equation}\label{BCs-cM}
x^{\al-\frac{n-1}2}\phi|_{\pd\cM}=f
\end{equation}
in the manifold if at each asymptotically AdS patch $\cU_j$ we have 
\[
x^{\al-\frac{n-1}2}\phi|_{x=0}=f_j\,.
\]
Notice that the variables $(t,\te)$ take values in
$\RR\times\SS^{n-2}$ and that $f$ can be understood as a function
defined on $\pd\cM$. In fact, we will use the notation
\[
\|f\|_{H^k(\pd\cM)}:=\sum_{j=1}^N\|f_j\|_{H^k_{t\te}}\,.
\]

We are now ready to prove that, given a smooth enough datum $f$ on the
conformal boundary, there is a unique solution to the Klein--Gordon
equation that satisfies the boundary condition~\eqref{BCs-cM}, and
that one can satisfactorily estimate this solution in terms of the
boundary datum. In the following theorem, we show how the solution
decomposes as the sum of terms that arise from the peeling-off procedure
described in Section~\ref{S.boundary} and another function that is, in a
way, smaller at infinity. We emphasize that, although the estimates
are presented in twisted Sobolev spaces, Theorem~\ref{T.Morrey}
allows us to convert them into pointwise bounds.

We will state this result for the case where there is
no source term, the boundary datum $f$ is supported in the region
$\{t>0\}$ and the equation has vanishing initial conditions, as this
is the case of direct interest in cosmology and the statement is
simpler. We are thus led to the problem
\begin{subequations}\label{LW-bdry}
\begin{gather}
\square_g\phi-\mu\phi=0\quad\text{in } \cM\,,\\
\phi|_{\cM_0}=0\,,\qquad \cT\phi|_{\cM_0}=0\,.
\end{gather}
\end{subequations}
Of course, adding general initial conditions (compatible with
the chosen boundary datum) and source terms is immediate in view of
Theorem~\ref{T.wave}.

\begin{theorem}\label{T.LW} 
Given a nonnegative integer $m$, let us take any integer
\begin{equation}\label{k>}
k>\frac{m+1+\al}2
\end{equation}
and a boundary datum $f=(f_1,\dots,f_N)$ as in~\eqref{datumf}, where
we assume that $f_j\in H^{m+2k}_{t\te}$ is supported in the region
$\{t>0\}$. Then the problem~\eqref{LW-bdry} has a unique solution,
which is in $\bH^{m+1}\loc(\cM)$ and can be written as
\[
\phi=\sum_{j=0}^k \psi_j\,.
\]
For $j<k$, the function $\psi_j(t,x,\te)$ is supported in the asymptotically AdS
patches $\bigcup_{i=1}^N\cU_i$ and reads as
\begin{equation}\label{asympt-uj}
\psi_j(t,x,\te)=x^{-\al+2j+\frac{n-1}2}\, \chi_j(x)\, H_j(t,\te)\,,
\end{equation}
where $\chi_j(x)$ is a smooth function that is identically $1$ in a
neighborhood of $x=0$ and 
\begin{equation}\label{bound-Hj}
\|H_j\|_{H^{m+2k-2j}_{t\te}}\leq C \|f\|_{H^{m+2k}(\pd\cM)}
\end{equation}
for some constant independent of $f$.  The function $\psi_k(t,x,\te)$
is in $\bH^{m+1}(\cM_{(-T,T)})$ for all $T$ and is bounded as
\begin{equation}\label{bound-phik}
\sum_{l=0}^{m+1}\|\cT^l \psi_k\|_{L^\infty\bH^{m+1-l}(\cM_{(-T,T)})}\leq C_T\|f\|_{H^{m+2k}(\pd\cM)}\,.
\end{equation}
\end{theorem}

\begin{proof}
Let us recall that, in an asymptotically AdS patch, the equations
$\Box_{g}\phi-\mu\phi=0$ and $P_gu=0$ are related by the
transformation~\eqref{phiu}. Therefore, by applying Theorem~\ref{T.large} in each asymptotically AdS chart of
$\cM$, we infer that there are functions $u_j$ of order
$\cO(x^{-\al+2j}f^{\leq 2j})$ supported in these
charts and such that the function
\[
\psi_k:=\phi-\sum_{j=0}^{k-1}x^{\frac{n-1}2}u_j
\]
satisfies the equation
\begin{equation*}
\square_g\psi_k-\mu\psi_k=x^{\frac{3-n}2}F
\end{equation*}
and the boundary condition
\[
x^{\al-\frac{n-1}2}\psi_k|_{\pd\cM}=0\,.
\]
Here $F$ is a function supported in the asymptotically AdS charts and
of order $\cO(x^{-\al+2k-2}f^{\leq 2k})$.  Notice that the functions
$\psi_j:=x^{\frac{n-1}2}u_j$ and $x^{\frac{3-n}2}F$ are obviously well
defined globally because $u_j$ is zero outside the asymptotically AdS
charts, and the expression $f^{\leq k}$ for the $N$-component function
$f$ has the obvious meaning. Furthermore, all the implicit functions
in the terms $\cO(\cdots)$ are uniformly bounded, and the fact that
the estimates~\eqref{asympt-uj} and~\eqref{bound-Hj} are satisfied is
an immediate consequence of Theorem~\ref{T.large} and the definition
of $\psi_j$.

Since $u_j=\cO(x^{-\al+2j}f^{\leq 2j})$ and $f$ is supported in the
set $\{t>0\}$, it follows from Theorem~\ref{T.large} that $u_j$ is
also supported in this set, so $\psi_k$ also satisfies the initial
condition
\[
\psi_k|_{\cM_0}=0\,,\qquad \cT\psi_k|_{\cM_0}=0\,.
\]
In an asymptotically AdS chart, the equation satisfied by $u_{k}$ reads
as
\[
P_g u_k=\tilde F
\]
with $\tilde F=\cO(x^{-\al+k-2}f^{\leq 2k})$.

Since the function $x^{-\al+2k-2}$ belongs to the Sobolev space
$\bH^m_x$ when the condition~\eqref{k>} holds,
an easy computation shows that the function $\tilde F$ belongs
to the space $\bH^m(\cM)$ with norm
\[
\|\tilde F\|_{\bH^m(\cM)}\leq C\|f^{\leq 2k}\|_{H^m_{t\te}}\leq C\|f\|_{H^{m+2k}_{t\te}}\,.
\]
Hence we can apply Theorem~\ref{T.wave} to show that there is a unique
solution $\psi_k$ to the above initial-boundary value problem, which
satisfies the bound~\eqref{bound-phik}.
\end{proof}

\section{Application to nonlinear wave equations}\label{Nonlinappl}
\label{S.NLW}




Our goal in this section is to prove the local well-posedness of the
nonlinear wave equation~\eqref{NLW} in an asymptotically AdS manifold
with a nontrivial boundary datum~$f$ at conformal infinity
(Eq.~\eqref{BCs2}). As before, we will assume that the nonlinearity is
of the form~\eqref{nonlinF}. To derive this result, which is stated
below as Theorem~\ref{T.nonlinear}, we will make use of the full range
of results obtained in the previous sections on the behavior of the
linear wave equation and the properties of twisted Sobolev spaces. 

Specifically, let us fix some positive number $T$ and consider the problem
\begin{subequations}\label{nonlinear}
\begin{gather}
\square_g\phi-\mu\phi= \Ga\, g(\nabla\phi,\nabla\phi)\quad\text{in } \cM_{(-T,T)}\,,\\
\phi|_{\cM_0}=0\,,\qquad \cT\phi|_{\cM_0}=0
\end{gather}
with the boundary condition
\begin{equation}
x^{\al-\frac{n-1}2}\phi|_{\pd\cM_{(-T,T)}}=f\,,
\end{equation}
\end{subequations}
with $\pd\cM_{(-T,T)}$ denoting the portion of the boundary $\pd\cM$
where $|t|<T$ and $f$ is as in Eq.~\eqref{datumf}. As before, $\Ga$ has the behavior
\[
\Ga=x^q \,\widehat\Ga(t,x,\te)
\]
in each asymptotically AdS chart, where the function $\widehat\Ga$ is smooth up to
the boundary, and the exponent $q$ can be different in each
asymptotically AdS chart. Just as in Theorem~\ref{T.LW}, we can safely
assume that $f$ is identically zero in the region $\{t\leq 0\}$.

The main result of this section is the following theorem, which proves that
the problem~\eqref{nonlinear} is locally well posed in suitable
twisted Sobolev spaces. Notice that the estimate~\eqref{bound-psij} below
makes sense because $\psi_j$ is supported in the asymptotically AdS
charts. By Theorem~\ref{T.Morrey}, the Sobolev estimates established
here immediately yield pointwise estimates for the solution. 

To prove this theorem, we start by peeling off the layers of the
solution that are large at infinity, and then use a suitable bootstrap
argument to control the remaining part of the solution. In the second
argument we need to use the properties of the twisted Sobolev spaces to
prove that certain nonlinear function of the solution is locally
Lipschitz continuous, which is done in Lemma~\ref{L.lipschitz} below.

\begin{theorem}\label{T.nonlinear}
Let us fix an integer $m>\frac n2$. Take nonnegative integers $k,l$
such that, together with the exponents $q$ of the nonlinearity~$\Ga$
  (which may be different in each asymptotically AdS chart), satisfy
  the following conditions:
\begin{enumerate}
\item $k>\dfrac{m+1+\al}2$.\\[0.3mm]
\item $l\geq 2k+m+\dfrac{n-1}2
$.\\[0.3mm]
\item $q>\al+m+\eta+\dfrac{5-n}2$, where $\eta$ is defined as in Proposition~\ref{P.multilinear}.
\end{enumerate}
Take a function $f\in H^l(\pd\cM_{(-T,T)})$. For all $T$ smaller than
some positive constant depending only on the norm
$\|f\|_{H^l(\pd\cM_{(-T,T)})}$, the problem~\eqref{nonlinear} has a
unique solution $\phi$, which can be written as
\[
\phi=\sum_{j=0}^k\psi_j\,.
\]
For $j\leq k-1$ and all $s$, the terms $\psi_j$ are supported in the
asymptotically AdS charts and are bounded as
\begin{align}\label{bound-psij}
\|x^{\al-2j-\frac{n-1}2}\psi_j\|_{W^{s,\infty}_x H^{m+1}_{t\te}} \leq C_s \|f\|_{H^l(\pd\cM_{(-T,T)})}\,,
\end{align}
while the term $\psi_k$ can be estimated as
\begin{equation}\label{bound-psik}
\sum_{l=0}^{m+1}\big\|\cT^l\psi_k\big\|_{L^\infty\bH^{m+1-l}(\cM_{(-T,T)})}\leq
C_T\,\|f\|_{H^l(\pd\cM_{(-T,T)})}\,.
\end{equation}
\end{theorem}

\begin{proof}
Let us start by peeling off the large behavior of the function
$\phi$ near the conformal boundary $\pd\cM_{(-T,T)}$.
Since in an asymptotically AdS patch the equations
$\square_g\phi-\mu\phi=0$ and $P_gu=0$ are related by the
transformation~\eqref{phiu}, we can invoke Theorem~\ref{T.large} in each asymptotically AdS chart of
$\cM$ to show that there are $k-1$ functions of the form
\[
\psi_j=x^{\frac{n-1}2}u_j\,,
\]
with $u_j$ of order $\cO\pol(x^{-\al+2j+\frac{n-1}2}f^{\leq 2j})$
and supported in these charts, such that the function
\[
\psi_k:=\phi-\sum_{j=0}^{k-1}\psi_j
\]
satisfies the equation
\begin{equation}\label{N-phik}
\square_g\psi_k-\mu\psi_k= \cN(\psi_k)+\rho_k
\end{equation}
and the boundary condition
\[
x^{\al-\frac{n-1}2}\psi_k|_{\pd\cM}=0\,.
\]
In this proof, all the implicit constants that appear in terms of
order $\cO\pol(x^rf^{\leq p})$ are bounded independently of $f$. The fact that the function $\psi_j$ satisfies~\eqref{bound-psij} is
immediate in view of Theorem~\ref{T.large}

Let us discuss the various terms that appear in
Eq.~\eqref{N-phik}. The term $\rho_k$ is supported in the
asymptotically AdS charts and of the form
\[
\rho_k=x^{\frac{n-1}2}\,\cO\pol(x^{-\al+2k-2} f^{\leq 2k})\,,
\]
so
\begin{align}\label{normrk}
\|\rho_k\|_{L^\infty\bH^{m}(\cM_{(-T,T)})}\leq C\|f\|_{H^l(\pd\cM_{(-T,T)})}\,.
\end{align}
When evaluated a point outside the asymptotically AdS charts, the
nonlinear term $\cN(\psi_k)$ looks simply as
\[
\cN(\psi_k) =\Ga\, g(\nabla\psi_k,\nabla \psi_k)\,,
\] 
that is, as the nonlinear term that we have introduced in the wave equation.
In an asymptotically AdS chart and setting $u_k:=x^{\frac{1-n}2}\psi_k$, Eq.~\eqref{N-phik} reads as
\[
P_gu_k=\tN(u_k)+\trho_k\,,
\]
where $\trho_k$ is of the
form
\[
\trho_k= \cO\pol(x^{-\al+2k-2}f^{\leq 2k})
\]
and the nonlinearity is
\begin{equation}\label{tN}
\tN(u_k):=Q(u_k,u_k)+Q(R_k,u_k)
\end{equation}
for some function 
\[
R_k=\sum_{j=0}^{k-1}\cO\pol(x^{-\al+2j}f^{\leq 2j})\,.
\]

To prove that the function $\psi_k$ is uniquely determined, we will
use a fixed point argument. For this, let us
consider the Banach space
\[
\cX:=L^\infty\bH^{m+1}(\cM_{(-T,T)})\cap W^{1,\infty}\bH^m(\cM_{(-T,T)})\,,
\]
endowed with its natural norm $\|\cdot\|_{\cX}$. The set of functions
in this space of norm less than $r$ will be denoted by
\[
\cX_r:=\big\{ \phi\in \cX:\|\phi\|_{\cX}<r\big\}\,.
\]
By Lemma~\ref{L.lipschitz} below, the inequality
\begin{equation}\label{eq-Lip1}
\|\cN(\psi)-\cN(\Psi)\|_{L^\infty\bH^m(\cM_{(-T,T)})}\leq C_1\,\|\psi-\Psi\|_{\cX}
\end{equation}
is valid for all $\psi,\Psi\in\cX_r$, with a constant that
depends on $r$ and $\|f\|_{H^l(\pd\cM_{(-T,T)})}$.

Let us now consider the solution map $S$ of the problem
\begin{gather*}
\square_g\phi-\mu\phi=F\quad\text{in } \cM_{(-T,T)}\,,\\
\phi|_{\cM_0}=0\,,\qquad \cT\phi|_{\cM_0}=0
\end{gather*}
with boundary condition
\[
x^{\al-\frac{n-1}2}\phi|_{\pd\cM}=0\,,
\]
which maps each $F\in L^\infty \bH^m(\cM_{(-T,T)})$ to the unique
solution $\phi=:S(F)$ of the problem. Theorem~\ref{T.wave} implies that $S$ defines a bounded map from
$L^\infty\bH^m(\pd\cM_{(-T,T)})$ to $\cX$ whose norm is at most
\begin{equation}\label{bound-S}
\|S\|_{L^\infty \bH^m(\cM_{(-T,T)})\to\cX}\leq C_2T
\end{equation}
for small $T$.

To construct the function $\psi_k$, let us recursively define the functions
\begin{align*}
\Psi_0&:=S\rho_k\,,\\
\Psi_i&:= S\rho_k+S\cN(\Psi_{i-1})\,,\qquad i\geq1\,.
\end{align*}
A standard argument then shows that if $T$ is smaller than some
$T_0(\|f\|_{H^l(\pd\cM_{(-T,T)})})$, then $\Psi_i$ converges to the
only solution $\psi_k$ to our problem, which satisfies the
bound~\eqref{bound-psik}. To see this, notice that the fact that
$\cN(0)=0$ and the estimate~\eqref{eq-Lip1} imply that
\[
\|\cN(\Psi)\|_{L^\infty \bH^m(\cM_{(-T,T)})} \leq C_1R
\]
for all $\Psi\in\cX_r$. Hence taking $r$ large enough (e.g., larger
than $2\|S\rho_k\|_{\cX}$) and using a standard bootstrap argument we
can ensure that
\begin{align*}
\|\Psi_i\|_{\cX}&\leq \|S\rho_k\|_{\cX}+ \|S\cN(\Psi_{i-1})\|_{\cX}\\
&\leq \frac r2+ C_2T \|\cN(\Psi_{i-1})\|_{L^\infty \bH^m(\cM_{(-T,T)})} 
\end{align*}
remains smaller than $r$ for small enough $T$ and all $i$. This allows
us to apply the inequality~\eqref{eq-Lip1} and the
bound~\eqref{bound-S} to estimate
\begin{align*}
\|\Psi_{i+1}-\Psi_i\|_{\cX}&=
\|S\cN(\Psi_{i})-S\cN(\Psi_{i-1})\|_{\cX}\\
&\leq C_1C_2T \|\Psi_{i}-\Psi_{i+1}\|_{\cX}\,.
\end{align*}
Choosing $T$ small enough so that $C_1C_2T<1$, the contraction mapping
theorem yields the desired function $\psi_k\in \cX$ as the limit of
$\Psi_i$ as $i$ tends to infinity. To conclude, the estimates for the
wave equation established in Theorem~\ref{T.wave} and the estimate for
the norm of $\rho_k$ (Eq.~\eqref{normrk} can then be easily bootstrapped
to obtain the bound~\eqref{bound-psik}.
\end{proof}

\begin{remark}\label{R.NLW}
  As required in Physics, the construction ensures that when the
  support of the function $f$ is contained in the set $t\geq T_0$,
  then the solution $\phi$ is identically zero for $t\leq T_0$ and,
  after becoming nonzero, is still a regular solution to the problem
  for a time $T$ that depends only on the norm of the boundary datum
  $f$. This time tends to infinity as the $H^l$ norm of $f$ tends to
  zero because in this case the Lipschitz constant $C_1$ above becomes
  arbitrarily small.
\end{remark}

To conclude, the following lemma contains the proof of the local Lipschitz
continuity of the function $\cN$, which we have used in the proof of
Theorem~\ref{T.nonlinear} above. We borrow the notation from the
demonstration of this result without further mention.

\begin{lemma}\label{L.lipschitz}
Under the hypotheses of Theorem~\ref{T.nonlinear}, the function $\cN$ satisfies the estimate
\[
\|\cN(\phi)-\cN(\psi)\|_{L^\infty\bH^m(\cM_{(-T,T)})}\leq C_0\,\|\phi-\psi\|_{\cX}
\]
for all $\phi,\psi\in\cX_r$, with the constant $C_0$ depending only on
$r$ and $\|f\|_{H^l(\pd\cM_{(-T,T)})}$.
\end{lemma}

\begin{proof}
When evaluated outside the asymptotically AdS charts, the function
$\cN$ is simply
\[
\cN(\phi)=\Ga\, g(\nabla\phi,\nabla\phi)
\]
and the estimate is standard. Therefore, let us evaluate $\cN$ in an
asymptotically AdS chart and introduce the notation
\[
u:=x^{\frac{1-n}2}\phi\,,\qquad v:=x^{\frac{1-n}2}\psi\,.
\]
The estimate for $\cN$ is then equivalent to showing
\begin{equation}\label{LiptN}
\|\tN(u)-\tN(v)\|_{L^\infty_t\bH^m}\leq C\,\|u-v\|_{\tX}
\end{equation}
for all $u,v\in\tX_r$, where the function 
\[
\tN(u,v)=Q(u,u)+Q(R_k,u)
\]
was defined in Eq.~\eqref{tN}. Here
$\tX:=L^\infty_t\bH^{m+1}\cap
W^{1,\infty}_t\bH^m$ and
\[
\tX_r:=\big\{ u\in\tX:\|u\|_{\tX}<r\big\}
\]
with the obvious definition of the norm.

To prove the inequality~\eqref{LiptN}, let us notice that the bilinear
function $Q$ can be written as
\begin{equation*}
Q(u,v)=x^{q+\frac{n-1}2}\bigg[c_{ij}\, \pd_iu\,
\pd_jv+\frac{c_i}x\,\big(\pd_i u\, v+u\, \pd_iv\big)+\frac{c_0}{x^2}uv\bigg]\,,
\end{equation*}
where the functions $c_{ij},c_i,c_0$ are smooth up to the boundary and
the indices $i,j$ (which are summed over) take values in the set
$\{t,x,\te^1,\dots,\te^n\}$. Since the twisted derivative
$\bD_x^{(r)}$ can be written as
\[
\bD_x^{(r)}=\sum_{r_1+r_2=r}b_{r_1r_2}x^{-r_1}\pd_x^{r_2}\,,
\]
one can check that
\begin{multline*}
\bD_x^{(r)}D_\te^sQ(u,v)=\sum_{r_1+r_2+r_3\leq r}\sum_{s_1+s_2\leq s}
x^{q+\frac{n-5}2-r_3}\bigg[\tilde c_{ij }\, \bD_x^{(r_1)}D_\te^{s_1}\pd_iu\,
\bD_x^{(r_2)}D_\te^{s_2} \pd_jv\\
+\frac{\tilde c_{i }}x\,\big(\pd_i \bD_x^{(r_1)}D_\te^{s_1} u\, \bD_x^{(r_2)}D_\te^{s_2} v+u\,
\pd_iv\big)+\frac{\tilde c_{0}}{x^2}\bD_x^{(r_1)}D_\te^{s_1}
u\, \bD_x^{(r_2)}D_\te^{s_2} v\bigg]
\end{multline*}
with some coefficients that are smooth up to the boundary. (In fact,
these coefficients also depend on the indices $r_1,r_2,s_1,s_2$, but
in order to keep the notation simple we are not making explicit this dependence.)

A simple application of Proposition~\ref{P.multilinear} then ensures
that 
\[
\|\bD_x^{(r)}D_\te^sQ(u,v)\|_{\bL^2}\leq C\|u\|_{\tX}\|v\|_{\tX}
\]
whenever $r+s\leq m$ and the exponent $q$ is larger than the quantity
specified in Theorem~\ref{T.nonlinear}. This shows
that
\begin{align}
\|Q(u,u)-Q(v,v)\|_{L^\infty_t\bH^m}&\leq
\|Q(u,u-v)\|_{L^\infty_t\bH^m}+ \|Q(v,u-v)\|_{L^\infty_t\bH^m}\notag \\
&\leq  C_r\,\|u-v\|_{\tX}\,. \label{Lip-Q}
\end{align}

An analogous argument can be used to take care of the term $Q(R_k,u)$,
finding that
\[
\|\bD_x^{(r)}D_\te^sQ(R_k,u)\|_{\bL^2}\leq C(\|f\|_{H^l(\pd\cM_{(-T,T)})})\,\|u\|_{\tX}
\]
under our hypothesis on the exponent~$q$.
By the definition of $\tN$, from this inequality and the
bound~\eqref{Lip-Q} we derive the desired estimate for $\tN$, thereby
completing the proof of the lemma.
\end{proof}

\section*{Acknowledgements}

A.E.\ is financially supported by the Ram\'on y Cajal program of the
Spanish Ministry of Science and thanks McGill University for hospitality
and support. A.E.'s research is supported in part by the Spanish MINECO under
grants~FIS2011-22566 and the ICMAT Severo
Ochoa grant SEV-2011-0087. The research of N.K.\ is supported by NSERC grant
RGPIN 105490-2011.

\bibliographystyle{amsplain}

\begin{thebibliography}{99}\frenchspacing
%

\bibitem{Anderson83}
M.T. Anderson, The Dirichlet problem at infinity for manifolds of negative curvature, J. Differential Geom. 18 (1983) 701--721. 

\bibitem{Anderson05}
M.T. Anderson, Geometric aspects of the AdS/CFT
correspondence, in: {\em AdS--CFT Correspondence: Einstein Metrics and
  their Conformal Boundaries}, ed. by O.~Biquard, IRMA Lect. Math. Theor. Phys. 8,
Eur. Math. Soc., Z\"urich, 2005, 1--31.

\bibitem{Anderson08}
M.T. Anderson, Einstein metrics with prescribed conformal infinity on
4-manifolds, Geom. Funct. Anal. 18 (2008) 305--366.

\bibitem{AS}
M.T. Anderson, R. Schoen, Positive harmonic functions on complete manifolds of negative curvature, Ann. of Math. 121 (1985) 429--461.

\bibitem{Bachelot1}
A. Bachelot, On wave propagation in the anti-de Sitter
cosmology, C. R. Math. Acad. Sci. Paris 349 (2011) 47--51.

\bibitem{Bachelot2}
A. Bachelot, The Klein--Gordon equation in the anti-de Sitter
cosmology, J. Math. Pures Appl. 96 (2011) 527--554.

\bibitem{Bachelot3}
A. Bachelot, New dynamics in the anti-de Sitter universe AdS$^5$,
Comm. Math. Phys. 320 (2013) 723--759.

\bibitem{BF}
P. Breitenlohner, D.Z. Freedman, Stability in gauged extended supergravity, Ann. Physics 144 (1982) 249--281.

\bibitem{Choquet1}
Y. Choquet-Bruhat, Solutions globales d'\'equations d'ondes sur l'espace-temps anti de Sitter, C. R. Acad. Sci. Paris S\'er. I Math. 308 (1989) 323--327.

\bibitem{Choquet2} 
Y. Choquet-Bruhat, Global solutions of Yang-Mills equations on anti-de Sitter spacetime, Classical Quantum Gravity 6 (1989) 1781--1789.

\bibitem{CC}
Y. Choquet-Bruhat, D. Christodoulou, Elliptic systems in $H^{s,\delta}$ spaces on manifolds which are Euclidean at infinity, Acta Math. 146 (1981) 129--150.

\bibitem{redshift}
M. Dafermos, I. Rodnianski, The red-shift effect and radiation decay on black hole spacetimes, Comm. Pure Appl. Math. 62 (2009) 859--919.

\bibitem{PRD}
A. Enciso, N. Kamran, Causality and the conformal boundary of
AdS in real-time holography, {Phys. Rev.~D} 85 (2012) 106016.

\bibitem{EK}
A. Enciso, N. Kamran, A Graham--Lee theorem for asymptotically AdS Einstein metrics, in preparation.

\bibitem{Evans}
L.C. Evans, {\em Partial differential equations}, American
Mathematical Society, Providence, 2010. 


\bibitem{GL}
C.R. Graham, J.M. Lee, Einstein metrics with prescribed conformal infinity on the ball, Adv. Math. 87 (1991) 186--225.

\bibitem{HT}
M. Henneaux, C. Teitelboim, Asymptotically anti-de Sitter spaces, Comm. Math. Phys. 98 (1985) 391--424. 

\bibitem{Holzegel}
G. Holzegel, Well-posedness for the massive wave equation on
asymptotically anti-de Sitter spacetimes, J. Hyperbolic Differ. Equ. 9
(2012) 239--261.

\bibitem{HW}
G. Holzegel, C. Warnick,
Boundedness and growth for the massive wave equation on asymptotically
anti-de Sitter black holes, {\tt arXiv:1209.3308}.

\bibitem{Lady}
O.A. Ladyzhenskaya, {\em The boundary value problems of mathematical physics}, Springer-Verlag, New York, 1985.

\bibitem{Maldacena}
J. Maldacena, The large $N$ limit of superconformal field theories and
supergravity, Adv. Theor. Math. Phys. 2 (1998) 231--252.

\bibitem{Nirenberg}
L. Nirenberg, H.F. Walker, The null
  spaces of elliptic partial differential operators in $\RR^n$,
  J. Math. Anal. Appl. 42 (1973) 271--303.

\bibitem{Rabier}
P.J. Rabier, Asymptotic behavior of the solutions of linear and quasilinear elliptic equations on $\RR^N$, Trans. Amer. Math. Soc. 356 (2004) 1889--1907.

\bibitem{Sullivan}
D. Sullivan, The Dirichlet problem at infinity for a negatively curved
manifold, J. Differential Geom. 18 (1983) 723--732.

\bibitem{Tataru}
D. Tataru, M. Tohaneanu, A local energy estimate on Kerr black hole backgrounds, Int. Math. Res. Not. 2011, 248--292.

\bibitem{Vasy}
A. Vasy, The wave equation on asymptotically anti de Sitter spaces, Anal. PDE 5 (2012) 81--144.

\bibitem{Warnick}
C. Warnick,
The massive wave equation in asymptotically AdS spacetimes,
Comm. Math. Phys. 321 (2013) 85--111.

\bibitem{Witten}
E. Witten, 
Anti de Sitter space and holography,
Adv. Theor. Math. Phys. 2 (1998) 253--291. 


\end{thebibliography}

\end{document}